\documentclass[11pt]{article}

\usepackage{mathtools}

\usepackage{booktabs}

\usepackage{verbatim}

\usepackage{amsfonts}

\usepackage{amsthm}

\usepackage{amssymb}

\usepackage{amsmath}

\usepackage{booktabs} 

\usepackage{enumerate}

\usepackage[all]{xy}

\usepackage{graphicx}

\usepackage{centernot}

\usepackage{stmaryrd}

\usepackage[pagebackref,colorlinks,linkcolor=red,citecolor=blue,urlcolor=blue,hypertexnames=true]{hyperref}

\usepackage[pagebackref,hypertexnames=true]{hyperref}

\newcommand{\N}{\mathbb{N}}

\newcommand{\Z}{\mathbb{Z}}

\newcommand{\R}{\mathbb{R}}

\newcommand{\C}{\mathbb{C}}

\newcommand{\K}{\mathcal{K}}

\newcommand{\epsin}{\in_{\epsilon}}

\newcommand{\Manoa}{M\=anoa}

\newcommand{\Hawaii}{Hawai\kern.05em`\kern.05em\relax i}

\newcommand{\ine}{\in_\epsilon}

\newcommand{\ses}{\subseteq_{\epsilon}}

\newcommand{\ind}{\in_\delta}

\DeclarePairedDelimiter\floor{\lfloor}{\rfloor}

\theoremstyle{plain}
\newtheorem{theorem}{Theorem}[section]
\newtheorem{lemma}[theorem]{Lemma}
\newtheorem{corollary}[theorem]{Corollary}
\newtheorem{proposition}[theorem]{Proposition}

\newtheorem{definition-theorem}[theorem]{Definition / Theorem}

\newtheorem*{conjecture*}{Conjecture}
\newtheorem*{theorem*}{Theorem}

\theoremstyle{definition}
\newtheorem{definition}[theorem]{Definition}

\theoremstyle{remark}
\newtheorem{remark}[theorem]{Remark}
\newtheorem{question}[theorem]{Question}

\newtheorem*{example*}{Example}  
\newtheorem*{remark*}{Remark}

\begin{document}
\title{Complexity rank for $C^*$-algebras}

\author{Arturo Jaime\footnote{University of \Hawaii~at \Manoa, 2565 McCarthy Mall, Keller 401A, Honolulu, HI 96816, USA; ajaime@hawaii.edu.} ~and Rufus Willett\footnote{University of \Hawaii~at \Manoa, 2565 McCarthy Mall, Keller 401A, Honolulu, HI 96816, USA; rufus@math.hawaii.edu.}}


\maketitle

\abstract{Complexity rank for $C^*$-algebras was introduced by the second author and Yu for applications towards the UCT: very roughly, this rank is at most $n$ if you can repeatedly cut the $C^*$-algebra in half at most $n$ times, and end up with something finite dimensional.  In this paper, we study complexity rank, and also a weak complexity rank that we introduce; having weak complexity rank at most one can be thought of as `two-colored local finite-dimensionality'.

We first show that for separable, unital, and simple $C^*$-algebras, weak complexity rank one is equivalent to the conjunction of nuclear dimension one and real rank zero.  In particular, this shows that the UCT for all nuclear $C^*$-algebras is equivalent to equality of the weak complexity rank and the complexity ranks for Kirchberg algebras with zero $K$-theory groups.  However, we also show using a $K$-theoretic obstruction (torsion in $K_1$) that weak complexity rank one and complexity rank one are not the same in general.

We then use the Kirchberg-Phillips classification theorem to compute the complexity rank of all UCT Kirchberg algebras: it is always one or two, with the rank one case occurring if and only if the $K_1$-group is torsion free.}

\pagebreak

\tableofcontents

\section{Introduction}

\subsection*{Background}

In recent work, the second author and Yu \cite{Willett:2021te} introduced the notion of decomposability of a $C^*$-algebra over a class of $C^*$-algebras.  This was motivated by two earlier ideas: the first of these was decomposability in coarse geometry (introduced by Guentner, Tessera, and Yu \cite{Guentner:2009tg,Guentner:2013aa}) and dynamics (introduced by Guentner, the second author, and Yu \cite{Guentner:2014bh}); the second was nuclear dimension (introduced by Winter and Zacharias \cite{Winter:2010eb}).  

Before going on with the general discussion, let us state the formal definition.  For a subset $S$ of a $C^*$-algebra $A$ and $a\in A$, write ``$a\in_\epsilon S$'' to mean that there is $s\in S$ with $\|a-s\|<\epsilon$.

\begin{definition}\label{intro dec}
Let $A$ be a unital $C^*$-algebra, and let $\mathcal{C}$ be a class of unital $C^*$-algebras.  Then $A$ \emph{decomposes} over $\mathcal{C}$ if for every finite subset $X$ of $A$ and every $\epsilon>0$ there exist $C^*$-subalgebras $C$, $D$, and $E$ of $A$ that are in the class $\mathcal{C}$ and contain $1_A$, and a positive contraction $h\in E$ such that:
\begin{enumerate}[(i)]
\item $\|[h,x]\|< \epsilon$ for all $x\in X$;
\item $hx\in_\epsilon C$, $(1_A-h)x\in_\epsilon D$, and $h(1_A-h)x\in_\epsilon E$ for all $x\in X$;
\item for all $e$ in the unit ball of $E$, $e\in_\epsilon C$ and $e\in_\epsilon D$.
\end{enumerate}
\end{definition}

In words, the definition says that one can use an almost central element ($h$ above) to locally cut the $C^*$-algebra $A$ into two pieces ($C$ and $D$ above) with well-behaved approximate intersection ($E$ above).

The main application of this notion is to the Universal Coefficient Theorem (UCT) of Rosenberg and Schochet \cite{Rosenberg:1987bh}.  For this paper we do not need any details about the UCT; suffice to say that the UCT is a $K$-theoretic property that a $C^*$-algebra may or may not have, and that whether or not the UCT holds for all nuclear $C^*$-algebras is an important open question.   The following theorem is the main result of \cite{Willett:2021te}.

\begin{theorem}\label{uct the}
If $A$ is a separable, unital $C^*$-algebra that decomposes over the class of nuclear UCT $C^*$-algebras, then $A$ itself is nuclear and satisfies the UCT.  

Moreover, all nuclear $C^*$-algebras satisfy the UCT if and only if any unital Kirchberg algebra\footnote{A unital $C^*$-algebra $A$ is a \emph{Kirchberg algebra} if it is separable, nuclear, and if for any non-zero $a\in A$ there are $b,c\in A$ with $bac=1_A$.} with zero $K$-theory decomposes over the class of finite-dimensional $C^*$-algebras.  
\end{theorem}

Due to the importance of the UCT, it thus becomes interesting to better understand the class of $C^*$-algebras that decompose over finite-dimensional $C^*$-algebras.  Inspired by this and coarse geometry \cite[Definition 2.9]{Guentner:2013aa}, the second author and Yu introduced a `complexity hierarchy' on $C^*$-algebras: we say a $C^*$-algebra has complexity rank zero if it is locally finite-dimensional\footnote{In the separable case, this is the same as being an AF $C^*$-algebra.}, and has complexity rank at most $n+1$ if it decomposes over the class of $C^*$-algebras of complexity rank at most $n$; note that having complexity rank at most one is then the same as decomposing over the class of finite-dimensional $C^*$-algebras.  One of our goals in this paper is to better understand the complexity rank for Kirchberg algebras, partly due to the connections to the UCT, and partly for the intrinsic interest of complexity rank as an invariant in its own right.

\subsection*{Results}

We first aim to make the connection between decomposability over the class of finite-dimensional $C^*$-algebras and nuclear dimension one more precise.  For this purpose, we introduce the notion of weak decomposability: this is the variant of Definition \ref{intro dec} below `with conditions on $E$ dropped'.  There is then a corresponding notion of weak complexity rank.  Let us spell out what this means for the weak complexity rank to be at most one.

\begin{definition}\label{intro dec}
Let $A$ be a unital $C^*$-algebra.  Then $A$ is of \emph{weak complexity rank at most one} if for every finite subset $X$ of $A$ and every $\epsilon>0$ there exist finite-dimensional $C^*$-subalgebras $C$ and $D$ of $A$ that contain $1_A$, and a positive contraction $h\in A$ such that:
\begin{enumerate}[(i)]
\item $\|[h,x]\|< \epsilon$ for all $x\in X$;
\item $hx\in_\epsilon C$ and $(1_A-h)x\in_\epsilon D$ for all $x\in X$.
\end{enumerate}
\end{definition}

We think of the pair $\{h,1_A-h\}$ of approximately central positive contractions in Definition \ref{intro dec} as being a `partition of unity', and we think of having weak complexity rank at most one as being `two-colored locally finite-dimensional'.

This notion turns out to be very closely related to nuclear dimension one.

\begin{theorem}\label{intro wcr}
For a separable, unital, simple\footnote{We establish much of this theorem without the simplicity assumption: see Theorem \ref{wcr1} below for details.}  $C^*$-algebra $A$, the following are equivalent:
\begin{enumerate}[(i)]
\item $A$ has nuclear dimension at most one, and real rank zero.
\item $A$ has weak complexity rank at most one.
\end{enumerate}
\end{theorem}

Having established Theorem \ref{intro wcr}, it is important to determine if weak complexity rank and complexity rank are actually the same: indeed, if they were, Theorem \ref{uct the} (plus the fact that all Kirchberg algebras have nuclear dimension one \cite[Theorem G]{Bosa:2014zr} and real rank zero \cite{Zhang:1990aa}) would imply the UCT for all nuclear $C^*$-algebras.  This question motivates the next theorem.

\begin{theorem}\label{intro tf}
Let $A$ be a unital $C^*$-algebra of complexity rank at most one.  Then $K_1(A)$ is torsion free.
\end{theorem}

As there are Kirchberg algebras with arbitrary countable $K$-theory groups \cite[Section 3]{Rordam:1995aa}, it follows from Theorems \ref{intro wcr} and \ref{intro tf} that complexity rank and weak complexity rank are different in general.  Whether they are equal in special cases is still interesting, however: Theorems \ref{uct the} and \ref{intro wcr} show that the UCT for all nuclear $C^*$-algebras is equivalent to equality of the weak and strong complexity ranks for Kirchberg algebras with zero $K$-theory.  

For general Kirchberg algebras, all we can say about the complexity rank is that it is at least one, and that it is at least two if the $K_1$-group has torsion.  If, however, we assume the UCT, and thus give ourselves access to the Kirchberg-Phillips classification theorem \cite{Kirchberg-ICM,Phillips-documenta}, then we get a complete computation.

\begin{theorem}\label{intro kirch}
All unital UCT Kirchberg algebras have complexity rank one or two.  Moreover, the rank one case occurs if and only if the $K_1$-group of the $C^*$-algebra is torsion free.
\end{theorem}

This theorem provides a striking contrast to the case of nuclear dimension / weak complexity rank, which are both always one for Kirchberg algebras.

\subsection*{Outline of the paper}

In Section \ref{dec sec} we discuss the main definitions, and give some reformulations of the main definitions (the version of decomposability used in this introduction is one of the stronger ones).  We also establish some consequences of weak complexity rank for nuclear dimension and existence of projections, and show that the complexity rank is subadditive on tensor products.

In Section \ref{wc1 sec} we study the class of $C^*$-algebras with weak complexity rank one in detail, and in particular establish Theorem \ref{intro wcr}.  Most of the section does not need anything beyond basic facts about nuclear dimension, as established in the seminal paper \cite{Winter:2010eb}.  However, the results going from weak complexity rank one to real rank zero are different: they use deep structure results for simple $C^*$-algebras from \cite{Winter:2012qf,Rordam:2004lw,Elliott:2011vc,Castillejos:2019aa}.  Moreover, some of the arguments used for this implication are due to the anonymous referee: see the acknowledgements below for details.

In Section \ref{cr1 sec} we use techniques from controlled $K$-theory as developed in \cite{Willett:2019aa} to establish Theorem \ref{intro tf}.  This and the results of the previous section allow us to distinguish weak complexity rank and complexity rank.  They will also be used for our results determining the complexity rank of UCT Kirchberg algebras in the next section.

In Section \ref{kirch sec}, we establish Theorem \ref{intro kirch}.  Our argument proceeds by adapting a technique developed by Enders \cite{Enders:2015aa} to estimate the nuclear dimension of Kirchberg algebras.  The Kirchberg-Phillips classification theorem \cite{Phillips-documenta,Kirchberg-ICM} is crucial here; we note that we need the existence and uniqueness theorems for morphisms that come as part of this (see Theorems \ref{kp stab the} and \ref{class k} below for the precise versions we use), not `only' the fact that UCT Kirchberg algebras are classified by $K$-theory.   Following Enders, we also need R\o{}rdam's crossed product models for Kirchberg algebras \cite{Rordam:1995aa}.  

Finally in the short Section \ref{question sec}, we list some natural questions.

\subsection*{Notation and conventions}

The symbol $A$ is reserved throughout for a $C^*$-algebra.  The unit of $A$ will be denoted $1$, or $1_A$ if there is risk of confusion.

Let $\epsilon>0$.  For $a,b\in A$, we write ``$a\approx_{\epsilon} b$'' if $\|a-b\|<\epsilon$.  For a subset $S$ of $A$ and $a\in A$, we write ``$a\epsin S$'' if there exists $s\in S$ such that $\|a-s\|< \epsilon$.  For subspaces $S$ and $T$ of $A$, we write ``$S\subseteq_\epsilon T$'' if for all elements $s$ of the unit ball of $S$, there exists $t$ in the unit ball of $T$ with $\|s-t\|<\epsilon$.  

For a $C^*$-algebra $A$, $A_1:=\{a\in A\mid \|a\|\leq 1\}$ is the closed unit ball, and $A_+:=\{a\in A\mid a\geq 0\}$ is the positive elements.  The multiplier algebra of a $C^*$-algebra $A$ is written $M(A)$. The symbol $\K$ denotes the compact operators on $\ell^2(\N)$.  For $C^*$-algebras $A$ and $B$, $A\otimes B$ is always the spatial (equivalently, minimal) tensor product.  For a unitary $u\in M(A)$, $\text{Ad}_u:A\to A$ denotes the conjugation automorphism defined by $a\mapsto uau^*$.

For a $C^*$-algebra $A$, $K_0(A)$ and $K_1(A)$ are its even and odd (topological) $K$-theory groups, and $K_*(A):=K_0(A)\oplus K_1(A)$ is the corresponding graded group; here `graded' means that the direct sum decomposition is remembered as part of the structure.  Homomorphisms $\alpha:K_*(A)\to K_*(B)$ will always be assumed to be graded, i.e.\ satisfying $\alpha(K_i(A))\subseteq K_i(B)$ for $i\in \{0,1\}$.  If $\phi$ is a $*$-homomorphism from $A$ to $B$, or an element of $KK_0(A,B)$, we write $\phi_*:K_*(A)\to K_*(B)$ for the induced (graded) homomorphism, and we will also use the same notation for the maps $\phi_*:K_i(A)\to K_i(B)$ for $i\in \{0,1\}$ that are defined by restricting (and co-restricting) $\phi_*$.

\subsection*{Acknowledgments}

We are grateful for support from the US NSF under DMS 1901522.  

The second author thanks Dominic Enders, Wilhelm Winter, and Guoliang Yu for conversations (in some cases, occurring some time ago) that influenced the results in this paper.

We also thank to the anonymous referee for a careful reading of the paper, and many useful comments.  In particular, the referee suggested the current proof of Lemma \ref{iota props} part \eqref{iota le} (which is much shorter than our original argument).  The referee also suggested Lemma \ref{ref lem} and its proof, and the proof of Proposition \ref{fmt rr0} that is based on this: in the first version of this paper we were only able to establish Proposition \ref{fmt rr0} under the additional assumption that $A$ has at most finitely many extreme tracial states, so the referee's ideas allowed for really significant improvements here.

\section{Definitions and basic properties}\label{dec sec}

In this section, we introduce the main definitions that we will study in this paper.

\begin{definition}\label{local in}
Let $\mathcal{C}$ be a class of $C^*$-algebras.  A $C^*$-algebra $A$ is \emph{locally} in $\mathcal{C}$ if for any finite subset $X$ of $A$ and any $\epsilon>0$ there is a $C^*$-subalgebra $C$ of $A$ that is in $\mathcal{C}$, and such that $x\in_\epsilon C$ for all $x\in X$.
\end{definition}

The following definition is a priori weaker than Definition \ref{intro dec}, and should be regarded as the `official' definition of what it means to decompose over a a class of $C^*$-algebras.  The two will be shown to be equivalent in Corollary \ref{good h 2} below.

\begin{definition}\label{ais}
Let $\mathcal{C}$ be a class of unital $C^*$-algebras.  A unital $C^*$-algebra $A$ \emph{decomposes over $\mathcal{C}$} if for every finite subset $X$ of $A$ and every $\epsilon>0$ there exist $C^*$-subalgebras $C$, $D$, and $E$ of $A$ that are in the class $\mathcal{C}$, and a positive contraction $h\in A$ such that:
\begin{enumerate}[(i)]
\item \label{ais com} $\|[h,x]\|< \epsilon$ for all $x\in X$;
\item \label{ais in} $hx\in_\epsilon C$, $(1-h)x\in_\epsilon D$, and $h(1-h)x\in_\epsilon E$ for all $x\in X$;
\item \label{ais e in} $E\ses C$ and $E\ses D$;
\item \label{ais mult} for all $e\in E_1$, $he\ine E$.
\end{enumerate}
\end{definition}

We now come to the fundamental definition for this paper.  

\begin{definition}\label{f c}
Let $\alpha$ be an ordinal number.
\begin{enumerate}[(i)]
\item If $\alpha=0$, let $\mathcal{D}_0$ be the class of unital $C^*$-algebras that are locally finite dimensional.
\item If $\alpha>0$, let $\mathcal{D}_\alpha$ be the class of unital $C^*$-algebras that decompose over $C^*$-algebras in $\bigcup_{\beta<\alpha} \mathcal{D}_\beta$.
\end{enumerate}
A unital $C^*$-algebra has \emph{finite complexity} if it is in $\mathcal{D}_\alpha$ for some $\alpha$, in which case its \emph{complexity rank} is the smallest possible $\alpha$.
\end{definition}

\begin{remark}\label{fdc rem}
Definition \ref{f c} is partly motivated by a notion of geometric complexity due to Guentner, Tessera, and Yu \cite[Definition 2.9]{Guentner:2013aa}.  In previous work of the second author and Yu \cite[Appendix A.2]{Willett:2021te}, we showed that if $X$ is a bounded geometry metric space then the geometric complexity of $X$ in the sense of \cite[Definition 2.9]{Guentner:2013aa} is an upper bound for the complexity rank of the uniform Roe algebra $C^*_u(X)$; there are other examples based on groupoid theory coming from \cite{Guentner:2014aa}.  We will not pursue this further here, however.
\end{remark}

We record three basic lemmas for use later in the paper.  The first two follow from straightforward transfinite inductions on $\alpha$ that we leave to the reader. 

\begin{lemma}\label{sums}
Let $A_1,...,A_n$ be unital $C^*$-algebras. Then for any ordinal $\alpha$, $A_1\oplus \cdots \oplus A_n$ is in $\mathcal{D}_\alpha$ if and only if each $A_i$ is in $\mathcal{D}_\alpha$. \qed
\end{lemma}

\begin{lemma}\label{quotients}
For any ordinal $\alpha$, the class $\mathcal{D}_\alpha$ is closed under taking quotient $C^*$-algebras. \qed
\end{lemma}

\begin{lemma}\label{loc in}
For any ordinal $\alpha$, any unital $C^*$-algebra that is locally in $\mathcal{D}_\alpha$ is in $\mathcal{D}_\alpha$.  Moreover, $\mathcal{D}_\alpha$ is closed under inductive limits with unital connecting maps. 
\end{lemma}

\begin{proof}
As the definitions are all local in nature, the fact that a $C^*$-algebra that is locally in $\mathcal{D}_\alpha$ is in $\mathcal{D}_\alpha$ is straightforward.  To see closure under inductive limits, note that by Lemma \ref{quotients} we may assume that the connecting maps in a given inductive system are injective.  Given this, the part on inductive limits follows from the part on local containment.  
\end{proof}

\subsection{Equivalent formulations}

In this subsection, we show that the definition of decomposability bootstraps up to stronger versions of itself.  We then use techniques of Christensen \cite{Christensen:1980kb} to show that the class of $C^*$-algebras of complexity rank at most one admits a particularly nice characterization.  

We need four very well-known lemmas; we record them for the reader's convenience as we will use them over and over again.

\begin{lemma}\label{spec lem}
Let $a$ and $b$ be bounded operators on a Hilbert space with $b$ normal.  Then the spectrum of $a$ is contained within distance $\|a-b\|$ of the spectrum of $b$.
\end{lemma}

\begin{proof}
We need to show that if $d(z,\text{spectrum}(b))>\|a-b\|$, then $a-z$ is invertible.  Indeed, in this case the continuous functional calculus implies that $\|(b-z)^{-1}\|< \|a-b\|^{-1}$.  Hence 
$$
\|(a-z)(b-z)^{-1}-1\|\leq \|(a-z)-(b-z)\|\|(b-z)^{-1}\|<1,
$$
whence $(a-z)(b-z)^{-1}$ is invertible, and so $a-z$ is invertible too.
\end{proof}  

\begin{lemma}\label{close lem}
Let $a\in A$ be an element in a $C^*$-algebra, let $\epsilon>0$, and let $B$ be a $C^*$-subalgebra of $A$ such that $a\in_\epsilon B$.  
\begin{enumerate}[(i)]
\item \label{pos part} If $a$ is positive, then there is positive $b\in B$ such that $\|b\|\leq \|a\|$ and $a\approx_{2\epsilon} b$.
\item \label{proj part} If $a$ is a projection and $\epsilon<1/2$, there is a projection $p\in B$ such that $a\approx_{2\epsilon} p$.
\end{enumerate}
\end{lemma}

\begin{proof}
For part \eqref{pos part}, let $b_0\in B$ be such that $a\approx_{\epsilon} b_0$.  Let $b_1=\frac{1}{2}(b_0+b_0^*)$, which is self-adjoint and still satisfies $b_1\approx_{\epsilon} a$.  Then $b_1$ has spectrum contained in $(-\epsilon,\|a\|+\epsilon)$ by Lemma \ref{spec lem}.  Hence if $f:\R\to\R$ is defined by 
$$
f(t):=\left\{\begin{array}{ll} 0 & -\infty<t\leq 0 \\ t & 0<t< \|a\| \\ \|a\| & \|a\|\leq t<\infty \end{array}\right.,
$$ 
then by the functional calculus $b:=f(b_1)$ is a positive contraction such that $b\approx_\epsilon b_0$.  Hence $a\approx_\epsilon b\approx_\epsilon b_0$ and we are done.

Part \eqref{proj part} is similar: this time $b_1$ chosen as above has spectrum contained in $(-\epsilon,\epsilon)\cup (1-\epsilon,1+\epsilon)$, and if $\chi$ is the characteristic function of $(1/2,\infty)$, then $p:=\chi(b_1)$ is a projection in $B$ such that $p\approx_{2\epsilon} a$.
\end{proof}

\begin{lemma}\label{alm proj}
Let $a$ be a self-adjoint element of a $C^*$-algebra $A$ such that $\|a^2-a\|< \epsilon\leq 1/4$.  Then there is a projection $p\in A$ such that $p\approx_{\sqrt{\epsilon}} a$.  
\end{lemma}

\begin{proof}
Let $t$ be in the spectrum of $a$.  Then $t(1-t)$ is in the spectrum of $a^2-a$, so $|t(1-t)|< \epsilon$.  Hence either $|t|< \sqrt{\epsilon}$, or $|1-t|< \sqrt{\epsilon}$, and so the spectrum of $a$ is contained in $(-\sqrt{\epsilon},\sqrt{\epsilon})\cup (1-\sqrt{\epsilon},1+\sqrt{\epsilon})$.  As $\sqrt{\epsilon}\leq 1/2$, the characteristic function $\chi$ of $(1/2,\infty)$ is continuous on the spectrum of $a$, and the functional calculus implies that $p:=\chi(a)$ is a projection that satisfies $p\approx_{\sqrt{\epsilon}}a$.  
\end{proof}

\begin{lemma}\label{proj under}
Let $A$ be a $C^*$-algebra.  Let $p,q$ be projections in $A$, and assume that $\|p-q\|< \epsilon\leq 1/4$.  Then there is a unitary $u$ in the unitization of $A$ (or in $A$ itself if it is already unital) such that $\|u-1\|< 10\epsilon$ and $p=uqu^*$.
\end{lemma}

\begin{proof}
Passing to the unitization of $A$ if necessary, we may assume $A$ is unital.  Let $v=(1-p)(1-q)+pq$.  Then one computes that $v-1=p(q-p)+(p-q)q$, so 
\begin{equation}\label{v-1}
\|v-1\|<2\epsilon.
\end{equation}  
As $2\epsilon<1$, $v$ is invertible.  Moreover, one checks that $vp=pq=qv$, so $vpv^{-1}=q$.  Hence also $v^*vp=v^*qv=(qv)^*v=(vp)^*v=pv^*v$ and so in particular $(v^*v)^{-1/2}$ commutes with $p$.  Let now $u:=v(v^*v)^{-1/2}$.  Then $u$ is unitary, and the previous computations show that $up=v(v^*v)^{-1/2}p=vp(v^*v)^{-1/2}=qv(v^*v)^{-1/2}=qu$, so $upu^*=q$.  Note moreover that 
$$
\|v^*v-1\|\leq \|v^*-1\|\|v\|+\|v-1\|<\epsilon(1+1+2\epsilon)= 2\epsilon(1+\epsilon) < 3\epsilon
$$
as $\epsilon\leq 1/4$.  Hence by the functional calculus 
$$
(1+3\epsilon)^{-1/2}\leq (v^*v)^{-1/2}\leq (1-3\epsilon)^{-1/2}
$$
and so by elementary estimates using that $\epsilon\leq 1/4$, $\|1-(v^*v)^{-1/2}\|\leq  4\epsilon$.  It follows from this and line \eqref{v-1} (which also implies that $\|v\|< 1+2\epsilon$) that 
$$
\|1-u\|\leq \|v\|\|1-(v^*v)^{-1/2}\|+\|1-v\|< (1+2\epsilon)4\epsilon+2\epsilon=10\epsilon
$$
as claimed.
\end{proof}

We hope the following lemma clarifies the definition of decomposability.

\begin{lemma}\label{good h}
Let $\mathcal{C}$ be a class of unital $C^*$-algebras.  A unital $C^*$-algebra $A$ decomposes over $\mathcal{C}$ if and only if it satisfies the following condition.  

For every finite subset $X$ of $A$ and every $\epsilon>0$ there exist $C^*$-subalgebras $C$, $D$, and $E$ of $A$ that are in the class $\mathcal{C}$, and a positive contraction $h\in A$ such that:
\begin{enumerate}[(i)]
\item $\|[h,x]\|< \epsilon$ for all $x\in X$;
\item $hx\in_\epsilon C$, $(1_A-h)x\in_\epsilon D$, and $h(1_A-h)x\in_\epsilon E$ for all $x\in X$;
\item $E\ses C$, $E\ses D$, and $1_E\in C\cap D$;
\item \label{long h} $h=h_E+p$ and $1_A-h=(1_E-h_E)+q$, where $h_E$ is a positive contraction in $E$, and $p\in C$ and $q\in D$ are projections that are orthogonal to $1_E$, and satisfy $1_A=1_E+p+q$.  
\end{enumerate}
\end{lemma}

Schematically, we thus have a spectral decomposition of $h$ that `looks like'\footnote{The picture is maybe slightly misleading in that $E$ need not be contained in the intersection $C\cap D$; however, this can be arranged if $\mathcal{C}$ is the class of finite-dimensional $C^*$-algebras as in Lemma \ref{d equiv} below.} the following.\\
\begin{center}
\includegraphics[width=8cm]{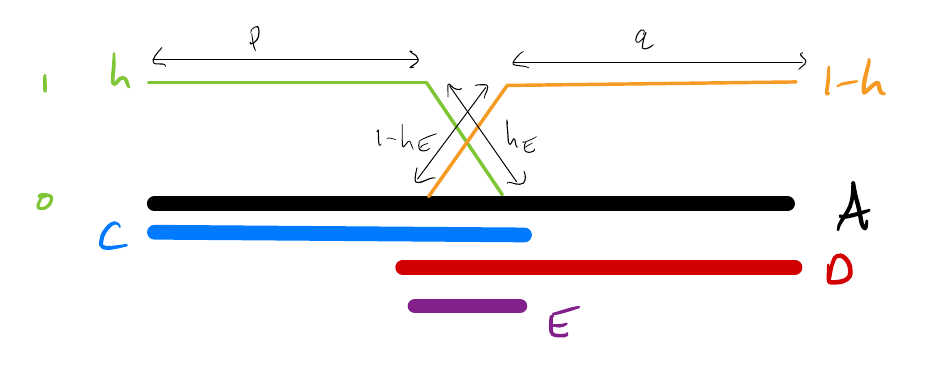}
\end{center}

\begin{proof}[Proof of Lemma \ref{good h}]
Assume first that $A$ satisfies the conditions from Lemma \ref{good h}.  Let $h$, $C$, $D$, and $E$ have the properties in Lemma \ref{good h} for a given $X$ and $\epsilon$; we claim they also satisfy the properties in Definition \ref{ais}.  Indeed, we need only check that for any $e\in E_1$, we have $he\ine E$.  For this, note that if $h=h_E+p$ with $h_E\in E$ and $p1_E=0$, then $he=(h_E+p)1_Ee=h_Ee$, which is (precisely) in $E$.

Conversely, assume $A$ satisfies the conditions from Definition \ref{ais}.  Let $\epsilon>0$, and let $X$ be a finite subset of $A$.  Let $\delta>0$, to be determined in the course of the proof in a way depending only on $\epsilon$.  Let $C$, $D$, and $E$ be $C^*$-algebras in $\mathcal{C}$ and $h$ be a positive contraction that have the properties in Definition \ref{ais} with respect to the finite set $X\cup\{1_A\}$ and $\delta$.  Throughout the proof, the notation ``$\delta_n$'' refers to a quantity that converges to zero as $\delta$ tends to zero, and that depends only on $\delta$.  

Now, as $1_E\in_\delta C$, Lemma \ref{close lem} part \eqref{proj part} gives $\delta_1$ and a projection $p_E\in C$ such that $\|p_E-1_E\|<\delta_1$.  Hence by Lemma \ref{proj under} there is a unitary $u\in A$ and $\delta_2>0$ such that $\|u-1_A\|<\delta_2$, and so that $u1_Eu^*=p_E$.  Similarly, there is a projection $q_E\in D$ and a unitary $v\in A$ such that $\|v-1_A\|<\delta_3$ for some $\delta_3$, and such that $v1_Ev^*=q_E$.  Hence replacing $C$ with $u^*Cu$ and $D$ with $v^*Dv$, we may assume that $C$, $D$, $E$ and $h$ satisfy the conditions in Definition \ref{ais} for $X\cup\{1_A\}$ and some $\delta_4>0$, and moreover that $1_E\in C\cap D$.

As $1_Eh1_E\in_\delta E$, Lemma \ref{close lem} gives a positive contraction $h_E\in E$ with $1_Eh1_E\approx_{2\delta} h_E$.  Moreover, as $h1_E\in_\delta E$, we have $(1_A-1_E)h1_E\approx_\delta 0$ and taking adjoints gives $1_Eh(1_A-1_E)\approx_\delta 0$.  Hence if we write $h_{E^\perp}:=(1_A-1_E)h(1_A-1_E)$ then 
$$
h\approx_{2\delta} 1_Eh1_E+(1_A-1_E)h(1_A-1_E)\approx_{2\delta} h_E+h_{E^{\perp}}.
$$
Replacing $h$ with $h_E+h_{E^\perp}$, we may assume $h$ is a sum of two positive contractions, one of which is in $E$, and one of which is orthogonal to $E$; in particular, $h$ multiplies $E$ into itself.  Note then that
$$
h(1_A-h)=h_E(1_E-h_E)-h_{E^\perp}^2+h_{E^\perp}
$$
and so $h_E(1_E-h_E)-h_{E^\perp}^2+h_{E^\perp}\in_{\delta_4} E$.  As $h_E(1_E-h_E)$ is in $E$, this implies that $h_{E^\perp}^2-h_{E^\perp}\in_{\delta_4} E$; however, $h_{E^\perp}^2-h_{E^\perp}$ is in $(1_A-1_E)A(1_A-1_E)$, so we get $h_{E^\perp}^2-h_{E^\perp}\approx_{\delta_4} 0$.  Assuming $\delta$ is small enough to ensure that $\delta_4<1/4$, Lemma \ref{alm proj} implies there is $\delta_5$ and a projection $p\in (1_A-1_E)A(1_A-1_E)$ such that $p\approx_{\delta_5}h_{E^\perp}$ 	  Now, as $h=h\cdot 1_A\in_{\delta_5} C$ and as $h_E\in E\subseteq_{\delta_4} C$, we have that there is $\delta_6$ such that $p\in_{\delta_6} C$.  As $1_E\in C$ and as $p$ is orthogonal to $1_E$, Lemma \ref{close lem} part \eqref{proj part} gives a projection $p_C\in (1_C-1_E)C(1_C-1_E)$ and $\delta_7>0$ such that $p_C\approx_{\delta_7} p$.  Hence Lemma \ref{proj under} gives a unitary $u\in (1_A-1_E)A(1_A-1_E)$ and $\delta_8$ such that $\|(1_A-1_E)-u\|<\delta_8$, and such that $up_Cu^*=p$.  Replacing $C$ by $(1_E+u)C(1_E+u^*)$, we may assume that $C$ contains $p$.  

On the other hand, we have $1_A-h=1_E-h_E+(1_A-1_E-p)$.  Write $q=(1_A-1_E-p)$.  Arguing analogously to the above, we also see that $q\in_{\delta_9} D$ for some $\delta_9$, and so that there exists a unitary $v\in (1_A-1_E)A(1_A-1_E)$ such that $\|(1_A-1_E)-v\|<\delta_{10}$ for some $\delta_{10}$ and such that $v^*qv\in D$.  Replacing $D$ by $(1_E+v)D(1_E+v^*)$ and taking the original $\delta$ small enough, we are done.
\end{proof}

We are now able to deduce that Definition \ref{ais} is equivalent to the definition of decomposability (Definition \ref{intro dec}) that we used in the introduction.

\begin{corollary}\label{good h 2}
Let $\mathcal{C}$ be a class of unital $C^*$-algebras that contains $\C$ and is closed under finite direct sums and under taking $*$-isomorphic $C^*$-algebras.  A unital $C^*$-algebra $A$ decomposes over $\mathcal{C}$ if and only if it satisfies the following condition.  

For every finite subset $X$ of $A$ and every $\epsilon>0$ there exist $C^*$-subalgebras $C$, $D$, and $E$ of $A$ that are in the class $\mathcal{C}$ and contain $1_A$, and a positive contraction $h\in E$ such that:
\begin{enumerate}[(i)]
\item $\|[h,x]\|< \epsilon$ for all $x\in X$;
\item $hx\in_\epsilon C$, $(1-h)x\in_\epsilon D$, and $h(1-h)x\in_\epsilon E$ for all $x\in X$;
\item $E\ses C$ and $E\ses D$.
\end{enumerate}
\end{corollary}

\begin{proof}
It is immediate that the condition in Corollary \ref{good h 2} implies the condition in Definition \ref{ais}. 

For the converse, given $X$ and $\epsilon$, let $h$, $C$, $D$, $E$ satisfy the conditions in Lemma \ref{good h}.  Define $E':=\C p\oplus E\oplus \C q$, $C':=\text{span}\{C,1_A\}$, and $D':=\text{span}\{D,1_A\}$.  Note that $E'$ is $*$-isomorphic to $E$ (if $p=q=0$), $E\oplus \C$ (if one of $p$ or $q$ is zero), or $E\oplus \C\oplus \C$ (if both $p$ and $q$ are non-zero).  Similarly, $C'$ (respectively $D'$) is isomorphic to $C$ or $C\oplus \C$ (respectively, $D$ or $D\oplus \C$); in all cases, $E'$, $C'$ and $D'$ are therefore still in $\mathcal{C}$.  Direct checks then show that $E'$, $C'$, $D'$ and $h$ satisfy the conditions in Corollary \ref{good h 2}. 
\end{proof}

In the remainder of this section, we show that complexity rank at most one bootstraps up to a stronger version of itself.  This will be useful for the results of Section \ref{cr1 sec} on torsion in $K_1$-groups.  For this, we need to recall a theorem of Christensen \cite[Theorem 5.3]{Christensen:1980kb} about perturbing almost inclusions of finite dimensional $C^*$-algebras to honest inclusions.

\begin{theorem}[Christensen]\label{chr the}
Let $A$ be a $C^*$-algebra, and let $E$ and $C$ be $C^*$-subalgebras of $A$ with $E$ finite-dimensional.  If $0< \epsilon \leq 10^{-4}$ and $E\ses  C$, then there exists a partial isometry $v\in A$ such that $\|v-1_E\|< 120\sqrt{\epsilon}$ and $vEv^*\subseteq C$. \qed
 \end{theorem}

\begin{proposition}\label{d equiv}
A unital $C^*$-algebra $A$ has complexity rank at most one if and only if it has the following property.  

For any finite subset $X$ of the unit ball of $A$ and any $\epsilon>0$ there exist finite-dimensional $C^*$-subalgebras $C$, $D$ and $E$ of $A$ that contain the unit and a positive contraction $h\in E$ such that:
\begin{enumerate}[(i)]
\item $\|[h,x]\|< \epsilon$ for all $x\in X$;
\item $hx\epsin C$, $(1_A-h)x\epsin D$ and $(1_A-h)hx\epsin E$ for all $x\in X$;
\item $E$ is contained in both $C$ and $D$.
\end{enumerate}
\end{proposition}

\begin{proof} 
Using Corollary \ref{good h 2}, a $C^*$-algebra $A$ with the property in the statement implies that has complexity rank at most one.  Assume then that $A$ has complexity rank at most one, and let $X$ be a finite subset of the unit ball of $A$, and let $\epsilon>0$.  Fix $\delta>0$, to be chosen by the rest of the proof in a way depending only on $\epsilon$.  Throughout the proof, anything called ``$\delta_n$'' for some $n$ is a positive constant that depends only on the original $\delta$, and tends to zero as $\delta$ tends to zero.  

Let $h_0$, $C_0$, $D_0$, and $E_0$ satisfy the conclusion of Lemma \ref{good h} for $X$ and $\delta$; in particular, then, each of $C_0$, $D_0$, and $E_0$ are unital and locally finite-dimensional $C^*$-subalgebras of $A$ (although not necessarily with the same unit as $A$) and we can write $h=h_{E_0}+p$, where $h_{E_0}\in E_0$ is a positive contraction, $p\in C_0$ is a projection that is orthogonal to $1_{E_0}$, and $q:=1_A-1_{E_0}-p$ is a projection in $D_0$.  

Choose a finite-dimensional $C^*$-subalgebra $E_1$ of $E_0$ that contains the unit $1_{E_0}$ of $E_0$ (whence $1_{E_0}$ is also the unit of $E_1$), and is such that $h(1_A-h)x\in_{2\delta}E_1$ for all $x\in X$, and such that $h_{E_0}\in_{2\delta} E_1$.  Choose a finite-dimensional $C^*$-subalgebra $C_1$ of $C_0$ such that $E_1\subseteq_{2\delta} C_1$, $hx\in_{2\delta} C_1$ for all $x\in X$, so that $p\in_{2\delta} C_1$ and so that $1_{E_0}\in_{2\delta} C_1$.  As $1_{E_0}\in _{2\delta} C_1$, Lemma \ref{close lem} gives a projection $p_{CE}\in C_1$ such that $\|1_{E_0}-p_{CE}\|<4\delta$.  As long as $\delta$ is suitably small, Lemma \ref{proj under} gives $\delta_1>0$ and a unitary $u\in A$ such that $\|u-1_A\|<\delta_1$ and so that $up_{CE}u^*=1_{E_0}$.  Define $C_2:=uC_1u^*$.  Then $1_{E_0}\in C_2$, and for some $\delta_2>0$, we have that $E_1\subseteq_{\delta_2} C_2$, $hx\in_{\delta_2} C_2$ for all $x\in X$, and that $p\in_{\delta_2} C_2$.  As $p\in_{\delta_2} (1_A-1_{E_0})C_2(1_A-1_{E_0})$, we similarly find a projection $p_C\in (1_A-1_{E_0})C_2(1_A-1_{E_0})$ and a unitary $v\in (1_A-1_{E_0})A(1_A-1_{E_0})$ such that for some $\delta_3>0$, $\|v-(1_A-1_{E_0})\|<\delta_3$, and such that $vp_Cv^*=p$.  Define $C_3:=(1_{E_0}+v)C_2(1_{E_0}+v^*)$.  Then $C_3$ is a finite-dimensional $C^*$-subalgebra of $A$ that contains $p$ and $1_{E_0}$, and such that there is $\delta_4>0$ such that $E_1\subseteq_{\delta_4} C_3$ and  $hx\in_{\delta_4} C_3$ for all $x\in X$.   Analogously, find a finite-dimensional $C^*$-subalgebra $D_3$ of $D_0$ that contains $q$ and $1_{E_0}$, and such that $E_1\subseteq_{\delta_4} D_3$ and $(1-h)x\in_{\delta_4} D_3$ for all $x\in X$.

Now, let $E_2$ be the (finite-dimensional) $C^*$-subalgebra of $A$ spanned by $E_1$ and $p$ and $q$, and let $C_4$ (respectively, $D_4$) be the (finite-dimensional) $C^*$-subalgebra of $A$ spanned by $C_3$ (respectively $D_3$) and $1_A$.  These $C^*$-algebras $E_2$, $C_4$, and $D_4$ satisfy the following conditions: all contain $1_{E_0}$, $p$ and $q$ (and therefore $1_A$); $E_2\subseteq_{\delta_4} D_4$ and $(1-h)x\in_{\delta_4} D_4$ for all $x\in X$; $E_2\subseteq_{\delta_4} C_4$ and $hx\in_{\delta_4} C_4$ for all $x\in X$; $h_{E_0}\in_{2\delta} E_2$.  Define $E:=E_2$ and use Lemma \ref{close lem} part \eqref{pos part} to choose a positive contraction $h_E$ in $1_{E_0}E1_{E_0}=E_1$ such that $h_E\approx_{4\delta} h_{E_0}$.

Now, using Theorem \ref{chr the}, if $\delta_4\leq 10^{-4}$ there exists a partial isometry $w_C\in A$ such that $w_CEw_C^*\subseteq C_4$, $w_C^*w_C=1_{E}$, and so that $\|w_C-1_{E}\|\leq 120 \sqrt{\delta_4}=:\delta_5$.  As $1_{E}=1_A$, $w_C$ must be unitary as long as $\delta$ is small enough that $120\sqrt{\delta_4}<1$.  Assuming this, define $C_5:=w_C^*C_4w_C$, so $C$ contains $E$, and satisfies $hx\in_{\delta_6} C$ for some $\delta_6>0$ and all $x\in X$.  Similarly, there is a unitary $w_D\in A$ such that $\|w_D-1_{E}\|\leq \delta_5$, and so that $w_DEw_D^*\subseteq D_4$.  Define $D:=w_D^*D_4w_D$.  At this point, the reader can check that the $C^*$-subalgebras $C$, $D$, and $E$ together with $h:=h_E+p$ satisfy the conditions in the statement of this proposition with respect to some $\delta_7>0$.  Taking the original $\delta$ suitably small, we are done.
\end{proof}

\subsection{Weak finite complexity}\label{wfc con sec}

The main motivation for introducing finite complexity is that it gives a sufficient condition for a $C^*$-algebra to satisfy the UCT.  In contrast, the weaker version that we introduce here does not obviously have any $K$-theoretic consequences.  Instead, we introduce it as it seems of some interest as a structural property in its own right, and as it serves as a bridge between complexity rank and some more established dimension notions for $C^*$-algebras like nuclear dimension and real rank; these relations will be explored in the rest of this subsection, and in Section \ref{wc1 sec} below.

\begin{definition}\label{d3}
Let $\mathcal{C}$ be a class of unital $C^*$-algebras.  A unital $C^*$-algebra $A$ \emph{weakly decomposes over $\mathcal{C}$} if for every finite subset $X$ of $A$ and every $\epsilon>0$ there exist $C^*$-subalgebras $C$ and $D$ of $A$ that are in the class $\mathcal{C}$, and a positive contraction $h\in A$ such that:
\begin{enumerate}[(i)]
\item $\|[h,x]\|< \epsilon$ for all $x\in X$;
\item $hx\in_\epsilon C$ and  $(1-h)x\in_\epsilon D$ for all $x\in X$.
\end{enumerate}
\end{definition}

In other words, weak decomposability is like decomposability, but with the conditions on the `intersection' $E$ dropped.

\begin{definition}\label{f c 2}
Let $\alpha$ be an ordinal number.
\begin{enumerate}[(i)]
\item If $\alpha=0$, let $\mathcal{WD}_0$ be the class of unital $C^*$-algebras that are locally finite-dimensional.
\item If $\alpha>0$, let $\mathcal{WD}_\alpha$ be the class of unital $C^*$-algebras that weakly decompose over $C^*$-algebras in $\bigcup_{\beta<\alpha} \mathcal{WD}_\beta$.
\end{enumerate}
A $C^*$-algebra $D$ has \emph{weak finite complexity} if it is in $\mathcal{WD}_\alpha$ for some $\alpha$, in which case its \emph{weak complexity rank} is the smallest possible $\alpha$.
\end{definition}

Clearly the weak complexity rank of a $C^*$-algebra is bounded above by its complexity rank.  We will see later in the paper (see Corollary \ref{wcr1 is not cr1}) that the two are different in general.  

In the remainder of this subsection, we discuss two basic consequences of weak finite complexity: the first gives a weak existence of projections property (see \cite{Blackadar:1994tx} for background), and the second gives bounds on nuclear dimension (see \cite{Winter:2010eb} for background).   

Here is the weak existence of projections property; see Subsection \ref{rr0 ss} below for a stronger conclusion under stronger hypotheses.    

\begin{lemma}\label{wd implies lp}
If $A$ is a unital $C^*$-algebra with finite weak complexity, then the span of the projections in $A$ is dense.
\end{lemma}

\begin{proof}
We proceed by transfinite induction on the weak complexity rank.  The base case is clear, so let $\alpha>0$ be an ordinal number and assume the result holds for all ordinals $\beta<\alpha$.  Let $a\in A$ be arbitrary, let $\epsilon>0$, and let $h$, $C$ and $D$ be as in the definition of weak decomposability with respect to $X=\{a\}$ and $\epsilon/3$.  Choose $c\in C$ and $d\in D$ with $\|ha-c\|< \epsilon/3$ and $\|(1-h)a-d\|< \epsilon/3$.  The inductive hypothesis implies that the span of the projections in $C$ and $D$ are dense, so each of $c$ and $d$ can be approximated within $\epsilon/6$ by a linear combination of projections. Hence $c+d$ can be approximated within $\epsilon/3$ by a linear combination of projections.  Putting this together with the fact that $\|a-(c+d)\|<2\epsilon/3$, we are done.
\end{proof}

It is shown in \cite[Lemma 7.3]{Willett:2021te} that $C^*$-algebras of finite complexity are always nuclear.  Here we give a variant of this result.  First we need to recall the definition of nuclear dimension from \cite[Definition 2.1]{Winter:2010eb}.  

\begin{definition}\label{fnd def}
A completely positive map $\phi:A\to B$ between $C^*$-algebras has \emph{order zero} if whenever $a,b\in A$ are positive elements such that $ab=0$, we have that $\phi(a)\phi(b)=0$. 

A $C^*$-algebra $A$ has \emph{nuclear dimension at most $n$} if for any finite subset $X$ of $A$ and any $\epsilon>0$ there exists a finite dimensional $C^*$-algebra $F$ and completely positive maps 
$$
\xymatrix{ A \ar[dr]_-\psi & & A \\ & F \ar[ur]_-\phi & }
$$
such that:
\begin{enumerate}[(i)]
\item $\phi(\psi(x))\approx_\epsilon x$ for all $x\in X$;
\item $\psi$ is contractive;
\item $F$ splits as a direct sum $F=F_0\oplus \cdots \oplus F_n$ and each restriction $\phi|_{F_i}$ is contractive and order zero. 
\end{enumerate}
\end{definition}

We recall a useful estimate of Pedersen, which is (a special case of) the main result of \cite{Pedersen:1993vh}.

\begin{lemma}\label{pci}
Let $a$ and $b$ be elements of a $C^*$-algebra with $b\geq 0$.  Then 
$$
\|[a,b^{1/2}]\|\leq \frac{5}{4}\|a\|^{1/2}\|[a,b]\|^{1/2}.\eqno\qed
$$
\end{lemma}

\begin{proposition}\label{loc fnd}
Let $\alpha$ be an ordinal number.
\begin{enumerate}[(i)]
\item \label{finite alpha} If $\alpha=n\in \N\cup\{0\}$, then any $C^*$-algebra in $\mathcal{WD}_n$ has nuclear dimension at most $2^n-1$;
\item \label{infinite alpha} In general, any $C^*$-algebra in $\mathcal{WD}_\alpha$ is locally in the class of $C^*$-algebras that are both in $\mathcal{WD}_\alpha$ and have finite nuclear dimension.
\end{enumerate}
\end{proposition}

\begin{proof}
We first establish part \eqref{finite alpha} by induction on $n$.  If $A$ belongs to $\mathcal{D}_0$, then it is locally finite dimensional, and this implies nuclear dimension zero: this is essentially contained in \cite[Remark 2.2 (iii)]{Winter:2010eb}, but we give an argument for the reader's convenience.  Let $X\subseteq A$ be a finite subset and let $\epsilon>0$.  Choose a finite-dimensional $C^*$-subalgebra $F$ of $A$ such that $x\in_{\epsilon} F$ for all $x\in X$.  Let $\psi:A\to F$ be any choice of conditional expectation (such exists by the finite-dimensional case of Arveson's extension theorem - see for example \cite[Theorem 1.6.1]{Brown:2008qy}), and let $\phi:F\to A$ be the inclusion $*$-homomorphism; it is straightforward to see that these maps have the right properties. 

Assume then that $N\geq1 $, and the result has been established for all $n<N$.  Let a finite subset $X$ of $A$ and $\epsilon>0$ be given; we may assume $X$ consists of contractions.  Let $C$ and $D$ be $C^*$-subalgebras of $A$ in some class $\mathcal{WD}_n$ for some $n< N$, and let $h\in A$ be a positive contraction as in the definition of weak decomposability with respect to the finite subset $X$ and parameter $\epsilon^2/(25\cdot 2^{2N})$.  The inductive hypothesis implies that $C$ and $D$ have nuclear dimension at most $2^{N-1}-1$.   Choose a set $X_C\subseteq C$ such that for each $x\in X$, there is $x_C\in X_C$ such that $\|hx-x_C\|<\epsilon/(4\cdot 2^N)$. Using finite nuclear dimension, choose completely positive maps $\psi_C:C\to F_C$ and $\phi_C:F_C\to C$ such that $\psi_C$ is contractive, such that $\phi_C(\psi_C(x))\approx_{\epsilon/8} x$ for all $x\in X_C$, and such that $F_C$ decomposes into $2^{N-1}$ direct summands such that the restriction of $\phi_C$ to each summand is contractive and order zero.  Let $X_D$, $\psi_D$, $\phi_D$, and $F_D$ have analogous properties with respect to $D$ and with $h$ replaced by $1-h$.

Now, using Arveson's extension theorem, we may extend each of $\psi_C$ and $\psi_D$ to contractive completely positive (ccp) maps defined on all of $A$ (we keep the same notation for the extensions).  Define $F:=F_C\oplus F_D$, and 
$$
\psi:A\to F,\quad a\mapsto \psi_C(h^{1/2}ah^{1/2})+\psi_D((1-h)^{1/2}a(1-h)^{1/2}),
$$
which is easily seen to be ccp.  Define moreover 
$$
\phi:F\to A,\quad (f_C,f_D)\mapsto \phi_C(f_C)+\phi_D(f_D).
$$
To show that $A$ has nuclear dimension at most $2^N-1$, it suffices to show that $\phi(\psi(x))\approx_\epsilon x$ for any $x\in X$; the remaining properties are easily verified.  First note that as $\|[h,x]\|<\epsilon^2/(25\cdot 2^{2N})$, we have that $\|[h^{1/2},x]\|<\epsilon/(4\cdot 2^N)$ and $\|[(1-h)^{1/2},x]\|<\epsilon/(4\cdot 2^N)$ by Lemma \ref{pci}.  Hence
\begin{align*}
\psi(x) & =\psi_C(h^{1/2}xh^{1/2})+\psi_D((1-h)^{1/2}x(1-h)^{1/2}) \\ & \approx_{\epsilon/(4\cdot 2^N)} \psi_C(hx)+\psi_D((1-h)x).
\end{align*}
Choose $x_C\in X_C$ and $x_D\in X_D$ such that 
\begin{equation}\label{hx c d}
\|hx-x_C\|<\epsilon/(4\cdot 2^N) \quad \text{and} \quad \|(1-h)x-x_D\|<\epsilon/(4\cdot 2^N)
\end{equation}  
so we get
$$
\psi(x)\approx_{\epsilon / (2\cdot 2^N)} \psi_C(x_C)+\psi_D(x_D).
$$
As $\|\phi\|\leq 2^N$, this implies that 
$$
\phi(\psi(x))\approx_{\epsilon/2} \phi(\psi_C(x_C)+\psi_D(x_D))=\phi_C(\psi_C(x_C))+\phi_D(\psi_C(x_D)).
$$
By choice of $\phi_C$ and $\psi_C$, we have that $\phi_C(\psi_C(x_C))\approx_{\epsilon/8} x_C$, and similarly for $x_D$, whence 
$$
\phi(\psi(x))\approx_{3\epsilon/4} x_C+x_D.
$$
Finally, using line \eqref{hx c d} and that $N\geq 1$, we see that $x_C+x_D\approx_{\epsilon/4} hx+(1-h)x=x$, and so $\phi(\psi(x))\approx_{\epsilon} x$ as required.

Part \eqref{infinite alpha} can be proved using transfinite induction: essentially the same argument as used above for case  \eqref{finite alpha} works.  
\end{proof}

\subsection{Tensor products}

In this subsection we establish a permanence result for the complexity rank of tensor products: see Proposition \ref{perm rem} below.  For readability, we just state the result for complexity rank; the analogous fact holds for weak complexity rank as well, with a (simpler) version of the same proof.

The key ingredient we need is a result of Christensen on inclusions of tensor products of nuclear $C^*$-algebras: it follows by combining \cite[Proposition 2.6 and Theorem 3.1]{Christensen:1980kb}.

\begin{proposition}[Christensen]\label{tens in}
Let $E$ and $C$ be $C^*$-subalgebras of a $C^*$-algebra $A$ such that $E\ses C$ for some $\epsilon>0$, and let $B$ be a $C^*$-algebra.  Assume moreover that $E$ and $B$ are nuclear.   Then $E\otimes B \subseteq_{6\epsilon} C\otimes B$. \qed
\end{proposition}

\begin{lemma}\label{fin dim tp}
Let $B$ be a nuclear and unital $C^*$-algebra, and assume that $A$ is a unital $C^*$-algebra that decomposes over some class $\mathcal{C}$ of nuclear and unital $C^*$-algebras.  Then $A\otimes B$ decomposes over the class of $C^*$-algebras $C\otimes B$ with $C$ in $\mathcal{C}$.
\end{lemma}

\begin{proof}
Let $X$ be a finite subset of $A\otimes B$, and let $\epsilon>0$.  Up to an approximation, we may assume every element of $X$ is a finite sum of elementary tensors.  Fix such a finite sum $x=\sum_{i=1}^n a_i\otimes b_i$ for each $x\in X$, and let $X_A$ be the finite subset of $A$ consisting of all the elements $a_i$ appearing in such a sum for some $x\in X$.  Let $M$ be the maximum of the sums $\sum_{i=1}^n \|b_i\|$ as $x$ ranges over $X$.  We claim that if $\delta=\min\{\epsilon/M,\epsilon/6\}$ and if $E$, $C$, and $D$ are $C^*$-subalgebras of $A$ in the class $\mathcal{C}$ and $h\in A$ is a positive contraction that satisfy the conditions in Lemma \ref{good h} with respect to $X_A$ and $\delta$, then $E\otimes B$, $C\otimes B$, $D\otimes B$, and $h\otimes 1_B$ satisfy the conditions in Definition \ref{ais} with respect to $X$ and $\epsilon$; this will suffice to complete the proof.

Let us check the conditions from Definition \ref{ais}.  For condition \eqref{ais com}, if $x=\sum_{i=1}^n a_i\otimes b_i$ is one of our fixed representations of an element of $X$, then 
$$
\|[h\otimes 1_B,x]\|\leq \sum_{i=1}^n \|[a_i,h]\|\|b_i\|<\delta\sum_{i=1}^n \|b_i\|<\epsilon
$$
by assumption on $\delta$.  For condition \eqref{ais in}, note that for $x=\sum_{i=1}^n a_i\otimes b_i \in X$ and each $i$, there is $c_i\in C$ with $ha_i\approx_\delta c_i$.  Hence 
$$
\Bigg\|(h\otimes 1_B)x-\sum_{i=1}^n c_i\otimes b_i\Bigg\|=\sum_{i=1}^n \|ha_i-c\| \|b_i\|<\epsilon
$$
by choice of $\delta$, and so $(h\otimes 1_B)x\in_{\epsilon} C\otimes B$.  Similarly, $(1_{A\otimes B}-h\otimes 1_B)x\in _{\epsilon} D\otimes B$ and $h\otimes 1_B(1_{A\otimes B}-h\otimes 1_B)x\in _{\epsilon} E\otimes B$ for all $x\in X$.  For condition \eqref{ais e in}, we have that $E\otimes B\ses C\otimes B$ and $E\otimes B\ses D\otimes B$ by choice of $\delta$, Proposition \ref{tens in}, and the assumption that $B$ and everything in $\mathcal{C}$ is nuclear.  Condition \eqref{ais mult} from Definition \ref{ais} follows as if $h$ and $E$ satisfy condition \eqref{long h} from Lemma \ref{good h}, then $he\in E$ for all $e\in E$, whence also $(h\otimes 1_B)e\in E\otimes B$ for all $e\in E\otimes B$.
\end{proof}

\begin{proposition}\label{perm rem}
If $A$ is in $\mathcal{D}_\alpha$ and $B$ is in $\mathcal{D}_\beta$, then $A\otimes B$ is in $\mathcal{D}_{\alpha+\beta}$.
\end{proposition}

\begin{proof}
We first assume $\alpha=0$ and proceed by transfinite induction on $\beta$.  The base case $\beta=0$ says that a tensor product of unital locally finite-dimensional $C^*$-algebras is unital and locally finite-dimensional, which is straightforward.  Assume $\beta>0$, and let $B$ be a $C^*$-algebra in $\mathcal{D}_\beta$.  Using Lemma \ref{loc fnd}, $B$ is nuclear.  Hence by Lemma \ref{fin dim tp} $A\otimes B$ decomposes over the class of $C^*$-algebras of the form $A\otimes C$, with $C\in \bigcup_{\gamma<\beta} \mathcal{D}_\gamma$.  The inductive hypothesis therefore implies that $A\otimes B$ decomposes over the class $\bigcup_{\gamma<\beta} \mathcal{D}_{\gamma}$, so $A\otimes B$ is in $\mathcal{D}_\beta$ by definition.

Now fix $\beta$, and proceed by transfinite induction on $\alpha$.  The base case $\alpha=0$ follows from the discussion above.  For $\alpha>0$, the inductive step follows directly from Lemma \ref{fin dim tp} just as in the case above, so we are done.
\end{proof}

\section{Weak complexity rank one}\label{wc1 sec}

In this section, we study $C^*$-algebras of weak complexity rank (at most) one.  Let us first recall a definition from \cite{Brown:1991gf}.

\begin{definition}\label{rr0 def}
A $C^*$-algebra $A$ has \emph{real rank zero} if any self-adjoint element of $A$ can be approximated arbitrarily well by  a self-adjoint element with finite spectrum
\end{definition}

The following theorem is our main goal in this section.

\begin{theorem}\label{wcr1}
Let $A$ be a separable, unital $C^*$-algebra with real rank zero and nuclear dimension at most one.  Then $A$ has weak complexity rank at most one.

Conversely, let $A$ be a separable, unital $C^*$-algebra with weak complexity rank at most one.  Then $A$ has nuclear dimension at most one.  If in addition $A$ is simple, then it has real rank zero.
\end{theorem} 

It is conceivable that weak complexity rank at most one implies real rank zero in general: this seems quite an interesting question for the reasons discussed  in Remark \ref{gen rr0 rem}.

\begin{remark}\label{wcr0 rem}
Weak complexity rank zero is the same as being locally finite dimensional by definition, and this is in turn equivalent to having nuclear dimension zero by a slight elaboration on \cite[Remark 2.2 (iii)]{Winter:2010eb}; moreover, locally finite-dimensional $C^*$-algebras are easily seen to have real rank zero.  Hence if one replaces ``at most one'' by ``one'' everywhere it appears in Theorem \ref{wcr1}, the theorem is still correct.
\end{remark}

\subsection{From nuclear dimension and real rank to weak complexity rank}

In this subsection, we establish the sufficient condition for a $C^*$-algebra to have weak complexity rank at most one from Theorem \ref{wcr1}.

Let us start by recalling the basic structure theorem for order zero maps from \cite[Theorem 2.3]{Winter:1009aa} (see also \cite[4.1.1]{Winter:2003kx} for the case of finite-dimensional domain, which is all we will actually use).

\begin{theorem}[Winter-Zacharias]\label{oz the}
Let $\phi:A\to B$ be an order zero ccp map between $C^*$-algebras with $A$ unital, and define $h:=\phi(1_A)$.  Let $M(C^*(\phi(A)))$ be the multiplier algebra of the $C^*$-subalgebra of $B$ generated by the image of $F$, and let $\{h\}'$ be the commutant of $h$.  Then there is a $*$-homomorphism $\pi:A\to M(C^*(\phi(A)))\cap \{h\}'$ such that $\phi(a)=h\pi(a)$ for all $a\in A$.
\end{theorem}

The following lemma gives a version of Theorem \ref{oz the} in the case of finite-dimensional domain and real rank zero codomain that allows us to assume $h$ has finite spectrum, at the price of introducing an approximation; see \cite[Lemma 2.4]{Winter:2005tb} for a similar result.

\begin{lemma}\label{rr0 oz}
Let $A$ be a $C^*$-algebra of real rank zero, and let $\phi:F\to A$ be an order zero ccp map from a finite-dimensional $C^*$-algebra $F$ into $A$.  Let $h_0:=\phi(1)$ and $\pi:A\to M(C^*(\phi(A)))\cap \{h_0\}'$ be as in Theorem \ref{oz the}.  Then  for any $\epsilon>0$ there exists a positive contraction $h\in A$ with finite spectrum, that commutes with the image of $\pi$, and that satisfies
$$
\|\phi(f)-h\pi(f)\|\leq \epsilon\|f\|
$$
for all $f\in F$. 
\end{lemma}

\begin{proof}[Proof of Lemma \ref{rr0 oz}]
Let $h_0:=\phi(1)$ and let $\pi:F\to M(C^*(\phi(F)))\cap \{h_0\}'$ be the homomorphism given by Theorem \ref{oz the} such that 
$$
\phi(f)=h_0\pi(f)
$$
for all $f\in F$.  Write $F=M_{n_1}(\C)\oplus \cdots \oplus M_{n_k}(\C)$, and let 
$$
\{e_{ij}^{(l)} \mid l\in \{1,...,k\},i,j\in \{1,...,n_l\}\}
$$ 
be a set of matrix units for $F$.  Define  $m_{ij}^{(l)}:=\pi(e_{ij}^{(l)})\in M(C^*(\phi(F)))$.  For each $l$, let $(b_\lambda^{(l)})$ be a net of positive contractions in $C^*(\phi(F))$ that converges to $m_{11}^{(l)}$ in the strict topology; for simplicity, we assume that the index set for all these nets is the same as $l$ varies.  Replacing each $b_\lambda^{(l)}$ with $m_{11}^{(l)}b_\lambda^{(l)}m_{11}^{(l)}$, we may assume that $b_\lambda^{(l)}\leq m_{11}^{(l)}$ for all $\lambda$ and all $l$.  Let $\lambda$ be large enough that $\|b_\lambda^{(l)} h_0b_\lambda^{(l)} - m_{11}^{(l)}h_0m_{11}^{(l)}\|<\epsilon/2$, which exists by strict convergence.  Note that $b_\lambda^{(l)} h_0b_\lambda^{(l)}$ is an element of the hereditary $C^*$-subalgebra $\overline{b_\lambda^{(l)} Ab_\lambda^{(l)}}$ of $A$.  Hence using that real rank zero passes to hereditary subalgebras (see \cite[Theorem 2.6, (iii)]{Brown:1991gf}), we may find a positive contraction $h_{11}^{(l)}\in \overline{b_\lambda^{(l)} Ab_\lambda^{(l)}}$ with finite spectrum such that $\|h_{11}^{(l)}-m_{11}^{(l)}h_0m_{11}^{(l)}\|<\epsilon$.  Define now
\begin{align}
	h:=\sum_{l=1}^k\sum_{j=1}^{n_l} m_{j1}^{(l)} h_{11}^{(l)} m_{1j}^{(l)}.
\end{align}
We claim that this $h$ has the right properties.

We have to show that:
\begin{enumerate}[(i)]
\item \label{roz1}the image of $\pi$ commutes with $h$; 
\item \label{roz2}$h$ has finite spectrum; 
\item \label{roz3}$\|\phi(f)-h\pi(f)\|\leq \epsilon\|f\|$ for all $f\in F$.
\end{enumerate}

Indeed, for \eqref{roz1}, note that for any $m_{ij}^{(l)}$,
\begin{align*}
m_{ij}^{(l)}h & = \sum_{k=1}^{n_l} m_{ij}^{(l)}m_{k1}^{(l)} h_{11}^{(l)} m_{1k}^{(l)}=m_{i1}^{(l)}h_{11}^{(l)}m_{1j}^{(l)}= \sum_{k=1}^{n_l} m_{k1}^{(l)} h_{11}^{(l)} m_{1k}^{(l)}m_{ij}^{(l)}\\ & =hm_{ij}^{(l)}.
\end{align*}
As the $m_{ij}^{(l)}$ span $\pi(F)$, this implies that $h$ commutes with $\pi(F)$, and thus that $\pi$ takes image in $\{h\}'$ as we needed.  

For \eqref{roz2}, note that as $h_{11}^{(l)}\in b_{\lambda}^{(l)}Ab_{\lambda}^{(l)}$, and as $b_\lambda^{(l)} \leq m_{11}^{(l)}$, we have that $h_{11}^{(l)}\leq m_{11}^{(l)}$.  Hence the elements $m_{j1}^{(l)} h_{11}^{(l)} m_{1j}^{(l)}$ are mutually orthogonal as $j$ and $l$ vary.  As $j$ varies, each $m_{j1}^{(l)} h_{11}^{(l)} m_{1j}^{(l)}$ is moreover unitarily equivalent to $h_{11}^{(l)}$, so has the same (finite) spectrum as this element. It follows that the spectrum of $h$ is the union of the spectra of the $h_{11}^{(l)}$ as $l$ varies, so finite.

For \eqref{roz3}, note that 
$$
\|\phi(f)-h\pi(f)\|=\|h_0\pi(f)-h\pi(f)\|\leq \|h_0-h\|\|\pi(f)\|\leq \|h-h_0\|\|f\|.
$$
Hence it suffices to prove that $\|h-h_0\|<\epsilon$.  For this, note that as $h$ commutes with $\pi(F)$ and as $h\leq \sum_{l=1}^k\sum_{j=1}^{n_l} m_{jj}^{(l)}$, we have that 
\begin{align*}
h & =\Big(\sum_{l=1}^k\sum_{j=1}^{n_l} m_{jj}^{(l)}\Big)h=\sum_{l=1}^k\sum_{j=1}^{n_l} m_{j1}^{(l)}m_{1j}^{(l)}h=\sum_{l=1}^k\sum_{j=1}^{n_l} m_{j1}^{(l)}hm_{1j}^{(l)} \\ & =\sum_{l=1}^k\sum_{j=1}^{n_l} m_{j1}^{(l)}m_{11}^{(l)}hm_{11}^{(l)}m_{1j}^{(l)}
\end{align*}
Hence 
$$
h-h_0=\sum_{l=1}^k\sum_{j=1}^{n_l} m_{j1}^{(l)} h_{11}^{(l)} m_{1j}^{(l)}-\sum_{l=1}^k\sum_{j=1}^{n_l} m_{j1}^{(l)} m_{11}^{(l)}h_0 m_{11}^{(l)}m_{1j}^{(l)}
$$
and so
$$
\|h-h_0\|= \Big\|\sum_{l=1}^k\sum_{j=1}^{n_l} m_{j1}^{(l)}(h_{11}^{(l)}-m_{11}^{(l)}h_0m_{11}^{(l)})m_{1j}^{(l)}\Big\|
$$
As the terms $m_{j1}^{(l)}(h_{11}^{(l)}-m_{11}^{(l)}hm_{11}^{(l)})m_{1j}^{(l)}$ are mutually orthogonal as $j$ and $l$ vary, this equals 
$$
\sup_{l,j}\|m_{j1}^{(l)}(h_{11}^{(l)}-m_{11}^{(l)}h_0m_{11}^{(l)})m_{1j}^{(l)}\|\leq \|h_{11}^{(l)}-m_{11}^{(l)}h_0m_{11}^{(l)}\|<\epsilon,
$$
and we are done.
\end{proof}

For the next result, let $A_\infty:= \prod_\N A / \oplus_\N A$ denote the quotient of the product of countably many copies of a $C^*$-algebra $A$ by the direct sum.  We identify $A$ with its image in $A_\infty$ under the natural diagonal embedding, and write $A_\infty\cap A'$ for the relative commutant.  More generally, if $(B_n)$ is a sequence of $C^*$-algebras, we also write $B_\infty:= \prod_\N B_n / \oplus_\N B_n$ for the associated quotient.  Given a bounded sequence of linear maps $\phi_n:A\to B_n$, we write $\overline{\phi}:A\to B_\infty$ for the map induced by the `diagonal map' $a\mapsto (\phi_1(a),\phi_2(a),\cdots)$.  Similarly, given a bounded sequence of linear maps $\phi_n:A_n\to B$, we write $\overline{\phi}:A_\infty\to B_\infty$ for the map induced on quotients by the map $\prod_\N A_n\to \prod_\N B$ defined by $(a_n)\mapsto (\phi_n(a_n))$.

\begin{proposition}\label{main nd1 rr0}
Let $A$ be a separable, unital $C^{*}$-algebra with real rank zero and nuclear dimension at most one. Then there exists a positive contraction $h \in A_{\infty} \cap A'$ and sequences $(C_{n})$ and $(D_{n})$ of finite-dimensional $C^{*}$-subalgebras of $A$ such that $ha \in C_{\infty}$ and $(1-h)a \in D_{\infty}$ for all $a \in A$.
\end{proposition}

\begin{proof}
Since $A$ is separable and of nuclear dimension at most one, by \cite[Theorem 3.2]{Winter:2010eb} there exists a sequence $(\psi_{n}, \phi_{n}, F_{n})$ where:
\begin{enumerate}[(i)]
\item each $F_{n}$ is a finite-dimensional $C^{*}$-algebra that decomposes as a direct sum $F_{n} = F_{n}^{(0)}\oplus F_{n}^{(1)}$; 
\item each $\psi_{n}$ is a ccp map $A \rightarrow F_{n}$ such that the induced map $\overline\psi: A \rightarrow F_{\infty}$ is order zero;
\item each $\phi_{n}$ is a map $F_{n}\rightarrow A$, such that the restriction $\phi_{n}^{(i)}$ of $\phi_{n}$ to $F_{n}^{(i)}$ is ccp and order zero for $i\in \{0,1\}$;
\item for all $a \in A$, $\phi_{n}\psi_{n}(a) \rightarrow a$ as $n \rightarrow \infty$. 
\end{enumerate}
For $i\in \{0,1\}$, we will also need to consider the (order zero, ccp) maps $\overline{\phi^{(i)}}:(F^{(i)})_\infty\to A_\infty$ induced from $\phi_{n}^{(i)}: F_{n} \rightarrow A$, and the canonical projection $*$-homomorphism $\kappa^{(i)}:F_\infty\to F_\infty^{(i)}$.

As for each $n$ the map $\phi_{n}^{(0)}: F_{n}^{(0)} \rightarrow A$ is ccp and order zero, by Theorem \ref{oz the} there exists a positive contraction $h_{n}^{(0)} \in A$ and a $*$-homomorphism $\pi_{n}^{(0)}:F_{n}^{(0)}\to M\Big(C^*\big(\phi_{n}^{(0)}(F_{n}^{(0)})\big)\Big)\cap \{h_{n}^{(0)}\}'$ such that $$
\phi_{n}^{(0)}(b)=h_{n}^{(0)}\pi_{n}^{(0)}(b)$$ for all $b \in F_{n}^{(0)}$.  As in \cite[Corollary 3.1]{Winter:1009aa}, the formula 
$$
\rho_{n}^{(0)}(f\otimes b):= f(h_{n}^{(0)})\pi_{n}^{(0)}(b)
$$
determines a $*$-homomorphism 
$$
\rho_{n}^{(0)}: C_{0}(0,1]\otimes F_{n}^{(0)} \rightarrow A.
$$
Similarly, we get a $*$-homomorphism $\rho_{n}^{(1)}: C_{0}(0,1]\otimes F_{n}^{(1)} \rightarrow A$.  Define $S_{n} := \rho_{n}^{(0)}(C_{0}(0,1]\otimes F_{n}^{(0)})$ and $R_{n} := \rho_{n}^{(1)}(C_{0}(0,1]\otimes F_{n}^{(1)})$.
	
As in (the proof of) \cite[Proposition A.1]{Willett:2019aa}, the element 
$$
h :=\overline{\phi^{(0)}} \circ \kappa^{(0)} \circ \overline{\psi} (1)
$$ 
is a positive contraction in $A_{\infty} \cap A'$, and has the property that for all $a\in A\subseteq A_\infty$, 
$$
ha=\overline{\phi^{(0)}} \circ \kappa^{(0)} \circ \overline{\psi}(a) \quad \text{and}\quad (1-h)a=\overline{\phi^{(1)}} \circ \kappa^{(1)} \circ \overline{\psi}(a).
$$
For each $n$, if $x\in C_0(0,1]$ is the identity function then
$$
\phi_n^{(0)}(F_n^{(0)})=\rho_{n}^{(0)}(x\otimes F_{n}^{(0)})\subseteq \rho_{n}^{(0)}(C_{0}(0,1]\otimes F_{n}^{(0)})=S_n;
$$
hence $ha \in S_{\infty}$ for all $a\in A$, and similarly  $(1-h)a \in R_{\infty}$ for all $a \in A$.  	
	
From Lemma \ref{rr0 oz}, since $A$ has real-rank zero, for each $n$ there exists a positive contraction $\eta_{n}^{(0)}\in A$ with finite spectrum that commutes with the image of $\pi_{n}^{(0)}$ and that satisfies 
\begin{align}\label{bn zero}
\|\phi_{n}^{(0)}(b)-\eta_{n}^{(0)}\pi_{n}^{(0)}(b)\|\leq \frac{\|b\|}{n}
\end{align} 
for all $b\in F_{n}^{(0)}$.  Let $\sigma_{n}^{(0)}: C_{0}(0,1]\otimes F_{n}^{(0)} \rightarrow A$ be the $*$-homomorphism determined on elementary tensors by $f \otimes b \mapsto f(\eta_{n}^{(0)}) \pi_{n}^{(0)}(b)$.  This factors through a finite-dimensional $C^*$-algebra as in the diagram below
$$
\xymatrix{ C_{0}(0,1]\otimes F_{n}^{(0)} \ar[d] \ar[r]^-{\sigma_{n}^{(0)}} & A \\ C\big(\text{spec}(\eta_{n}^{(0)})\big)\otimes F_{n}^{(0)} \ar[ur] &  } 
$$ 
and so the image of $\sigma_{n}^{(0)}$ is a finite-dimensional $C^{*}$-subalgebra of $A$.  Define $C_{n}$ to be the image of $\sigma_n^{(0)}$, and let $C_\infty:= \prod_\N C_{n} / \oplus_\N C_{n}$ denote the corresponding $C^{*}$-subalgebra of $A_{\infty}$.  Working instead with $i=1$, we choose $\eta_n^{(1)}$ and use it to define $\sigma_n^{(1)}$, $D_n$, and $D_\infty$ precisely analogously.  

Let $a \in A \subseteq A_{\infty}$ and denote by $b:=\kappa^{(0)}\circ \overline{\psi} (a) \in (F^{(0)})_{\infty}$. Choose a sequence $(b_n)$ in $\prod_\N F_n^{(0)}$ that lifts $b$ and that satisfies $\|b_n\|\leq \|a\|$ for all $n$.  For a sequence $(a_n)$ in $\prod_\N A_n$, let us write $[(a_n)]$ for the corresponding element of $A_\infty$.  Then we compute that in  $A_{\infty}$
\begin{align*}
ha - [(\eta_{n}^{(0)}\pi_{n}^{(0)}(b_{n}))] &= [\phi_{n}^{(0)}(b_{n})] - [(\eta_{n}^{(0)}\pi_{n}^{(0)}(b_{n}))] \\
&= [(\phi_{n}^{(0)} - \eta_{n}^{(0)}\pi_{n}^{(0)})(b_{n})].
\end{align*} 
Line \eqref{bn zero} implies that 
\begin{align*}
\| (\phi_{n}^{(0)} - \eta_{n}^{(0)}\pi_{n}^{(0)})(b_{n}) \| \leq \frac{\|b_{n}\|}{n} \leq \frac{\|a\|}{n} \longrightarrow 0 \;\;\;\text{ as } \;n \longrightarrow \infty.
\end{align*}
Hence $ha = [(\eta_{n}^{(0)}\pi_{n}^{(0)}(b_{n}))]=[\sigma_n^{(0)}(x\otimes b_n)] \in C_{\infty}$.  A similar argument shows that $(1-h)a \in D_{\infty}$ and we are done.
\end{proof}

From Proposition \ref{main nd1 rr0} we have the following.

\begin{theorem}\label{weak cr1 the}
If $A$ is a separable, unital $C^{*}$-algebra with real rank zero and nuclear dimension at most one, then $A$ has weak complexity rank at most one.
\end{theorem}

\begin{proof}
Let $(C_n)$, $(D_n)$ and $h$ be as in the conclusion of Proposition \ref{main nd1 rr0}.  Lift $h$ to a positive contraction $(h_n)$ in $\prod_\N A_n$.  Then one checks that directly that for any finite subset $X$ and $\epsilon>0$ there is $N$ so that for all $n\geq N$, $C_n$, $D_n$, and $h_n$ satisfy the conditions needed for weak complexity rank at most one.
\end{proof}

The following corollary gives an interesting class of $C^*$-algebras with weak complexity rank one that we will use later.  For the statement, recall that a unital $C^*$-algebra $A$ is a \emph{Kirchberg algebra} if it is separable, nuclear, and if for any non-zero $a\in A$, there exist $b,c\in A$ such that $bac=1_A$ (note that this last condition implies simplicity).  See for example \cite[Chapter 4]{Rordam:2002cs} for background on this class of $C^*$-algebras.

\begin{corollary}\label{kirch wcr1}
Any unital Kirchberg algebra has weak complexity rank one.
\end{corollary}

\begin{proof}
Kirchberg algebras have real rank zero by the main result of \cite{Zhang:1990aa} and nuclear dimension one by \cite[Theorem G]{Bosa:2014zr}, whence weak complexity rank at most one by Theorem \ref{weak cr1 the}.  Kirchberg algebras do not have weak complexity rank zero as they are not locally finite dimensional. 
\end{proof}

\subsection{From weak complexity rank to real rank}\label{rr0 ss}

We now establish a partial converse to Theorem \ref{weak cr1 the}.  First, recall that Proposition \ref{loc fnd} shows that if $A$ has weak complexity rank at most one, then it has nuclear dimension at most one.  To establish a converse to Theorem \ref{weak cr1 the}, we therefore need to show that weak complexity rank at most one implies real rank zero.  We can do this for simple (separable, unital) $C^*$-algebras, but not in general; moreover, the proofs of our main result (see Proposition \ref{fmt rr0} below) are not self-contained, but rely on deep structural results for simple nuclear $C^*$-algebras.   Some key ideas in this section are due to the anonymous referee: in our first version of this paper, we also assumed that $A$ has at most finitely many extreme tracial states in Proposition \ref{fmt rr0} below.

We have generally tried to explain the properties we use as we need them: the most glaring omission is probably any discussion of $\mathcal{Z}$-stability, which we just use as a black box.

\begin{proposition}\label{fmt rr0}
Let $A$ be a simple, separable, unital $C^*$-algebra with weak complexity rank at most one.  Then $A$ has real rank zero.
\end{proposition}  

To establish this, we will need some facts about Cuntz (sub)equivalence and its interaction with tracial states.  We will recall the facts we need; we recommend \cite{Rordam:1992tx} and \cite{Ara:2011ww} for further background on these topics.  

Let $A$ be a $C^*$-algebra, and let $A\otimes \K$ be its stabilization.  Let $A$ be represented faithfully on a Hilbert space $H$, and use the corresponding representation of $A\otimes \K$ on $H\otimes \ell^2(\N)$ to identify elements of $A\otimes \K$ with (certain) $\N$-by-$\N$ indexed matrices with values in $A$.   Let $M_\infty(A)$ be the dense $*$-subalgebra of $A\otimes \K$ consisting of matrices with only finitely many non-zero entries, and identify $M_\infty(A)$ with the $*$-algebraic direct limit of the system 
\begin{equation}\label{mat incl}
M_n(A)\to M_{n+1}(A),\quad a\mapsto \begin{pmatrix} a & 0 \\ 0 & 0 \end{pmatrix}.
\end{equation}
Note that $M_\infty(A)$ is closed under functional calculus in $A\otimes \K$.  

Let now $(A\otimes \K)_+$ denote the positive elements in $A\otimes \K$, and let $M_\infty(A)_+$ denote the positive elements in $M_\infty(A)$.  For $a,b\in (A\otimes \K)_+$, we say $a$ is \emph{Cuntz subequivalent} to $b$, and write $a\lesssim b$, if there is a sequence $(r_n)$ in $A\otimes \K$ such that $r_nbr_n^*$ converges in norm to $a$.  We say $a$ and $b$ are \emph{Cuntz equivalent}, and write $a\sim b$, if $a\lesssim b$ and $b\lesssim a$.  Note that $\lesssim$ is a transitive and reflexive relation, and $\sim$ is an equivalence relation.  

Fix a (spatially induced) isomorphism $\phi:M_2(\K)\to \K$, and for $a,b\in (A\otimes \K)_+$, define
$$
a\oplus b:=(\text{id}_A\otimes \phi)\Bigg(\begin{pmatrix} a & 0 \\ 0 & b \end{pmatrix}\Bigg).
$$
As any two isomorphisms $\phi,\psi:M_2(\K)\to \K$ are conjugate by a unitary multiplier of $\K$, the Cuntz equivalence class of $a\oplus b$ does not depend on the choice of $\phi$.  

We record some basic properties of Cuntz subequivalence in the following lemma.  For a self-adjoint element $a$ in a $C^*$-algebra, let us write $a_+$ for its positive part.  Note that for any positive element $a$ in a $C^*$-algebra and any $\epsilon>0$, $(a-\epsilon)_+$ is in the original $C^*$-algebra, and not just in its unitization.  

\begin{lemma}\label{cse props}
Let $A$ be a $C^*$-algebra, and let $a,b\in (A\otimes \mathcal{K})_+$ and $x\in A\otimes \K$.  The following hold:
\begin{enumerate}[(i)]
\item \label{cuntz leq} If $a\leq b$, then $a\lesssim b$.
\item \label{cuntz trace} For any $\epsilon\geq 0$, $(xx^*-\epsilon)_+\sim (x^*x-\epsilon)_+$.
\item \label{cuntz sum} $a+b\lesssim a\oplus b$, and if $a$ and $b$ are orthogonal, then $a\oplus b\lesssim a+b$.
\item \label{cuntz close} If $\|a-b\|\leq \epsilon$, then $(a-\epsilon)_+\lesssim b$.
\end{enumerate}
\end{lemma}

\begin{proof}
Part \eqref{cuntz leq} follows from \cite[Lemma 2.3]{Rordam:1992tx} or \cite[Lemma 2.8]{Ara:2011ww}.  For part \eqref{cuntz trace}, let $x=u|x|$ be the polar decomposition of $x$ in the double dual $(A\otimes \K)^{**}$ (compare for example \cite[III.5.2.16]{Blackadar:2006eq}).  Then for any $\epsilon>0$, $y:=u(x^*x-\epsilon)_+^{1/2}$ is in $A\otimes \K$, and we have $y^*y=(x^*x-\epsilon)_+$ and $yy^*=(xx^*-\epsilon)_+$.  We have $y^*y\sim yy^*$ by \cite[Corollary 2.6]{Ara:2011ww}, completing the argument for \eqref{cuntz trace}.  Part \eqref{cuntz sum} follows from \cite[Lemma 2.10]{Ara:2011ww}, and part \eqref{cuntz close} follows  from \cite[Proposition 2.2]{Rordam:1992tx} or \cite[Theorem 2.13]{Ara:2011ww}.
\end{proof}

The next lemma is the only place in this subsection where the assumption of weak complexity rank at most one is used.

\begin{lemma}\label{decom trace}
Let $A$ be a unital $C^*$-algebra with weak complexity rank at most one.  Then for any $a\in M_\infty(A)_+$ and $\epsilon>0$ there is a projection $p\in M_\infty(A)$ such that 
$$
(a-\epsilon)_+\lesssim p\lesssim a\oplus a .
$$
\end{lemma}

\begin{proof}
Fix $n$ so that $a$ is in $M_n(A)$.  Note that $M_n(A)$ also has weak complexity rank at most one by (an easy variant of) Proposition \ref{perm rem}.  Let $\delta=\epsilon/10$.  The definition of weak complexity rank at most one and Lemma \ref{pci} give a positive contraction $h\in M_n(A)$ and finite-dimensional $C^*$-subalgebras $C,D$ of $M_n(A)$ such that $h^{1/2}ah^{1/2}\in_\delta C$ and $(1-h)^{1/2}a(1-h)^{1/2}\in_\delta D$, and such that $\|[a,h^{1/2}]\|<\delta$ and $\|[a,(1-h)^{1/2}]\|<\delta$.  Lemma \ref{close lem} then gives positive contractions $c\in C$ and $d\in D$ such that $\|h^{1/2}ah^{1/2}-c\|<2\delta$ and $\|(1-h)^{1/2}a(1-h)^{1/2}-d\|<2\delta$.  Note that 
$$
a\approx_{2\delta} h^{1/2}ah^{1/2}+(1-h)^{1/2}a(1-h)^{1/2}\approx_{4\delta} c+d\approx_{4\delta} (c-2\delta)_++(d-2\delta)_+
$$
whence by Lemma \ref{cse props} part \eqref{cuntz close}, $(a-10\delta)_+\lesssim (c-2\delta)_++(d-2\delta)_+$.  Hence  by part \eqref{cuntz sum} of Lemma \ref{cse props} 
\begin{equation}\label{lh inq}
(a-10\delta)_+\lesssim (c-2\delta)_+\oplus (d-2\delta)_+.
\end{equation}
On the other hand, as $\|c-h^{1/2}ah^{1/2}\|<2\delta$, whence part \eqref{cuntz close} of Lemma \ref{cse props} again gives $(c-2\delta)_+\lesssim h^{1/2}ah^{1/2} \lesssim a$ (where the second Cuntz subequivalence is clear from the definition).  Similarly, $(d-2\delta)_+\lesssim a$.  Combining these observations with line \eqref{lh inq} gives   
$$
(a-10\delta)_+\lesssim (c-2\delta)_+\oplus (d-2\delta)_+\lesssim a\oplus a.
$$
Now, $(c-2\delta)_+$ is contained in a finite-dimensional $C^*$-algebra, so has finite spectrum, whence it is Cuntz equivalent to its support projection $p_C\in C$; similarly $(d-2\delta)_+$ is Cuntz equivalent to its support projection $p_D$.  Setting $p:=p_C\oplus p_D$, we are done.
\end{proof}

We need some more terminology.  Let $A$ be a unital $C^*$-algebra, and let $T(A)$ be its tracial state space.  We equip $T(A)$ with its weak-$*$ topology, so it is a compact convex (possibly empty) subset of the unit ball of the dual space $A^*$ of $A$.  For any $\tau\in T(A)$, we abuse notation and also write $\tau$ for the map 
$$
\tau:(A\otimes \mathcal{K})_+\to [0,\infty],\quad (a_{ij})\mapsto \sum_{i\in \N}\tau(a_{ii})
$$
(here we use our fixed identification of elements of $A\otimes \K$ with $\N$-by-$\N$ matrices over $A$); the definition of $\tau:(A\otimes \mathcal{K})_+\to [0,\infty]$ depends only on the original element of $T(A)$ and not on the choice of identification.

For $\epsilon>0$, let $f_\epsilon:[0,\infty)\to [0,1]$ be the continuous function which is zero on $[0,\epsilon/2]$, $1$ on $[\epsilon,\infty)$ and linear on $[\epsilon/2,\epsilon]$.  For $a\in (A\otimes \mathcal{K})_+$, we define a function 
\begin{equation}\label{pair map}
\widehat{a}:T(A)\to [0,\infty],\quad \tau\mapsto \lim_{\epsilon\to 0} \tau(f_\epsilon(a))
\end{equation}
(the limit exists as the net $(\tau(f_\epsilon(a))_{\epsilon>0}$ is increasing as $\epsilon$ tends to zero).  Note that if $a\in M_\infty(A)_+$ then $\widehat{a}$ is finite-valued, and affine (as it is a pointwise limit of affine functions).  It need not be continuous in general, but if $p\in M_\infty(A)_+$ is a projection then $\widehat{p}$ is continuous, as then $f_{\epsilon}(p)=p$ for all $\epsilon\leq 1$.  
 
Lemma \ref{iota props} below records the properties of the maps $\widehat{a}$ that we will need.

\begin{lemma}\label{iota props}
Let $A$ be a unital $C^*$-algebra, let $a,b\in (A\otimes \K)_+$, and let $\widehat{a},\widehat{b}:T(A)\to [0,\infty]$ be as in line \eqref{pair map} above.  
\begin{enumerate}[(i)]
\item \label{iota le} If $a\lesssim b$, then $\widehat{a}\leq \widehat{b}$.
\item \label{iota sum} $\widehat{a\oplus b}=\widehat{a}+\widehat{b}$.
\end{enumerate}
\end{lemma}

\begin{proof}
Part \eqref{iota le} is well-known, but we could not find an exact statement in the literature\footnote{Compare \cite[Theorem II.2.2]{Blackadar:1982tg} or \cite[Proposition 2.1]{Cuntz:1978vy} for closely related results.}, so give an argument here for the reader's convenience; we thank the referee for providing the current much shorter version.  Assume that $a\lesssim b$.  According to the condition in \cite[Proposition 2.4 (iv)]{Rordam:1992tx} for any $\epsilon>0$ there exists $\delta>0$ and $r\in M_\infty(A)$ such that $f_\epsilon(a)=rf_\delta(b)r^*$.  Define $x:=f_\delta(b)^{1/2}r^*$, and note that $x^*x=f_{\epsilon}(a)$ and $xx^*$ is in the hereditary subalgebra $\overline{bAb}$ generated by $b$.  Hence for any $\tau\in T(A)$, using that $\|xx^*\|=\|f_\epsilon(a)\|\leq 1$, we have
\begin{equation}\label{fepsfdelt}
\tau(f_\epsilon(a))=\tau(x^*x)=\tau(xx^*)\leq \|\tau|_{\overline{bAb}}\|.
\end{equation}
On the other hand for any $c\in \overline{bAb}$ we have 
$$
c=\lim_{\delta\to 0}f_\delta(b)^{1/2}cf_{\delta}(b)^{1/2},
$$
whence for any positive $c\in \overline{bAb}$,
$$
\tau(c)=\lim_{\delta\to 0}\tau(f_\delta(b)^{1/2}cf_{\delta}(b)^{1/2})\leq \|c\|\lim_{\delta\to 0}\tau(f_\delta(b))=\|c\|\widehat{b}(\tau).
$$ 
Hence $\|\tau|_{\overline{bAb}}\|\leq \widehat{b}(\tau)$\footnote{The inequality ``$\widehat{b}(\tau)\leq \|\tau|_{\overline{bAb}}\|$'' is easily seen to hold so in fact one has equality here.} and combining this with line \eqref{fepsfdelt} implies $\tau(f_\epsilon(a))\leq \widehat{b}(\tau)$.  Taking the limit as $\epsilon\to 0$ gives $\widehat{a}(\tau)\leq \widehat{b}(\tau)$, and as $\tau$ was arbitrary, we are done.

Part \eqref{iota sum} is straightforward from the fact that $f_\epsilon(a\oplus b)=f_\epsilon(a)\oplus f_{\epsilon}(b)$ for any $a$, $b$ and $\epsilon$.  
\end{proof}

Variants of the following lemma are probably well-known.  

\begin{lemma}\label{put in minfinity}
Let $A$ be a $C^*$-algebra, and let $a,b\in A$ be positive elements such that $\|a^{1/2}ba^{1/2}-a\|<\epsilon$.  Then there exists  $x\in A$ such that $(a-\epsilon)_+=x^*x$, and $xx^*$ is in $b^{1/2}Ab^{1/2}$.
\end{lemma}

\begin{proof}
Using \cite[Lemma 2.2]{Kirchberg:2002ph} (or see \cite[Theorem 2.13]{Ara:2011ww}) there is $d\in A$ such that $da^{1/2}ba^{1/2}d^*=(a-\epsilon)_+$.  The element $x:=b^{1/2}a^{1/2}d^*$ has the desired property.
\end{proof}

The following lemma was communicated to us by the referee.  For the statement, let $A$ be a unital $C^*$-algebra with tracial state space $T(A)$, and let $\text{Aff}(T(A))$ denotes the space of continuous affine functions from $T(A)$ to $\R$.  We equip $\text{Aff}(T(A))$ with the supremum norm.  As already noted, if $p\in M_\infty(A)_+$ is a projection, then $\widehat{p}$ is an element of $\text{Aff}(T(A))$.  

\begin{lemma}\label{ref lem}
Let $A$ be a separable, unital, simple $C^*$-algebra with $T(A)$ non-empty, and with weak complexity rank at most one.  For any strictly positive element $\alpha\in \text{Aff}(T(A))$, either there exists a projection $q\in M_\infty(A)$ such that $\widehat{a}=q$, or there exists a projection $p\in M_\infty(A)$ such that  $\alpha-\widehat{p}$ is strictly positive and $\|\alpha-\widehat{p}\|\leq \frac{3}{4}\|\alpha\|$.  
\end{lemma}

\begin{proof}
Note first that $A$ has finite nuclear dimension by Proposition \ref{loc fnd}, whence in particular it is exact.  If $\mathcal{Z}$ is the Jiang-Su algebra, then as $A$ is simple, unital and has finite nuclear dimension, it is $\mathcal{Z}$-stable by \cite[Corollary 6.3]{Winter:2012qf}.   We may thus apply \cite[Proposition 2.6]{Castillejos:2019aa}\footnote{This is in turn based heavily on several results from \cite{Elliott:2011vc}.  The reader is referred to \cite[Section 1]{Castillejos:2019aa} for explanations of the terminology and notation used in \cite[Proposition 2.6]{Castillejos:2019aa}, which in particular contains enough information to explain why that result is applicable in our setting.} to conclude that there is a positive contraction $a\in A\otimes \mathcal{K}$ such that $\widehat{a}=\frac{1}{2}\alpha$.  For each $n$, the function 
$$
\phi_n:T(A)\to [0,\infty],\quad \tau\mapsto \tau(f_{1/n}(a))
$$
is continuous, and the sequence $(\phi_{n})$ is increasing and converges pointwise to $\widehat{a}$ by definition of the latter function.  Hence by Dini's theorem, $(\phi_n)$ converges uniformly to $\widehat{a}$.  Note that $f_{1/n}(a)\leq f_{\delta}((a-\epsilon)_+)$ for any $\delta\leq 1/2n$ and any $\epsilon\leq 1/2n$.  Hence 
$$
\phi_{n}\leq \widehat{(a-\epsilon)_+}\leq \widehat{a}
$$
whenever $\epsilon\leq 1/2n$, and so the net $(\widehat{(a-\epsilon)_+})_{\epsilon>0}$ in $\text{Aff}(T(A))$ also converges uniformly to $\widehat{a}$ as $\epsilon\to 0$.  Choose $\epsilon>0$ such that 
\begin{equation}\label{a and alpha}
\|\widehat{a}-\widehat{(a-\epsilon)_+}\|\leq \tfrac{1}{4}\|\alpha\|.
\end{equation}

Now, if the spectrum of $a$ is contained in $\{0\}\cup [\epsilon/2,1]$ then $q_0:=f_{\epsilon/4}(a)$ is a projection in $A$ such that $\widehat{q_0}=\widehat{a}$.  Moreover, Lemma \ref{iota props} part \eqref{iota sum} implies that if $q:=q_0\oplus q_0$, then $\widehat{q}=2\widehat{q_0}=2\widehat{a}=\alpha$, and we are done.  

Assume then that the spectrum of $a$ intersects $(0,\epsilon/2)$ non-trivially.  Let $g\in C_0((0,\epsilon/2))$ be a non-negative function such that $g(t)\leq t$ for all $t$, and so that $g$ is non-zero somewhere on the spectrum of $a$.  Note that the function $\widehat{g(a)}:T(A)\to [0,\infty]$ is finite-valued as $g(a)\leq a$, whence $\widehat{g(a)}\leq \widehat{a}$ by Lemma \ref{iota props} part \eqref{iota le}.  As $g(a)$ is non-zero, as $A$ is simple and as the kernel of any trace is an ideal, we have moreover that $\widehat{g(a)}:T(A)\to [0,\infty)$ is strictly positive.  We also have that 
$$
(a-\epsilon/2)_+\oplus g(a)\lesssim (a-\epsilon/2)_++ g(a)\lesssim a,
$$
where the first subequivalence uses that $g(a)$ and $(a-\epsilon/2)_+$ are orthogonal and Lemma \ref{cse props} part \eqref{cuntz sum}, and the second subequivalence uses that $(a-\epsilon/2)_++ g(a)\leq a$ and Lemma \ref{cse props} part \eqref{cuntz leq}.  Hence Lemma \ref{iota props} parts \eqref{iota le} and \eqref{iota sum} imply that $\widehat{a}-\widehat{(a-\epsilon/2)_+}\geq \widehat{g(a)}$, and so in particular 
\begin{equation}\label{sp}
\widehat{a}-\widehat{(a-\epsilon/2)_+}
\end{equation}
is strictly positive.

Now, for each $n$ let $p_n$ be the unit of $M_n(A)\subseteq A\otimes \mathcal{K}$.  As the sequence $(p_n)$ is an approximate unit for $A\otimes \mathcal{K}$, Lemma \ref{put in minfinity} (with $b=p_n$ for some large enough $n$) gives $n$ and $x\in A\otimes \mathcal{K}$ such that $x^*x=(a-\epsilon/2)_+$ and $xx^*\in p_n(A\otimes\mathcal{K})p_n=M_n(A)$.  Lemma \ref{decom trace} implies there is a projection $p\in M_\infty(A)$ such that 
\begin{equation}\label{x and p}
(xx^*-\epsilon/2)_+\lesssim p \lesssim xx^*\oplus xx^*.
\end{equation}
We claim this $p$ has the property in the statement.  

Indeed, as $x^*x=(a-\epsilon/2)_+$, Lemma \ref{cse props} part \eqref{cuntz trace} (in the special case $\epsilon=0$) implies that 
\begin{equation}\label{aex}
(a-\epsilon/2)_+\sim xx^*.
\end{equation}
Lines \eqref{x and p} and \eqref{aex}, and Lemma \ref{iota props} parts \eqref{iota le} and \eqref{iota sum}, imply that 
$$
\widehat{p}\leq 2\widehat{(a-\epsilon/2)_+}.
$$ 
Recalling also that $2\widehat{a}=\alpha$ and rearranging, we get
$$
2(\widehat{a}-\widehat{(a-\epsilon/2)_+})\leq \alpha-\widehat{p},
$$
whence $\alpha-\widehat{p}$ is strictly positive as the element in line \eqref{sp} has that property.  

On the other hand, as $(a-\epsilon)_+=((a-\epsilon/2)_+-\epsilon/2)_+=(x^*x-\epsilon/2)_+$,  Lemma \ref{cse props} part \eqref{cuntz trace} implies that 
\begin{equation}\label{aex2}
(a-\epsilon)_+\sim (xx^*-\epsilon/2)_+.
\end{equation}  
Lines \eqref{x and p} and \eqref{aex2}, and Lemma \ref{iota props} part \eqref{iota le} imply that $\widehat{(a-\epsilon)_+}\leq \widehat{p}$ whence 
$$
\alpha-\widehat{p}\leq \alpha-\widehat{(a-\epsilon)_+}.
$$
Hence using line \eqref{a and alpha} and that $\widehat{a}=\frac{1}{2}\alpha$ 
$$
\|\alpha-\widehat{p}\|\leq \|\alpha-\widehat{(a-\epsilon)_+}\|\leq \tfrac{1}{2}\|\alpha\|+\|\widehat{a}-\widehat{(a-\epsilon)_+}\|\leq \tfrac{3}{4}\|\alpha\|,
$$
and we are done.
\end{proof}

We are now ready for the proof of Proposition \ref{fmt rr0}, which was communicated to us by the referee.

\begin{proof}[Proof of Proposition \ref{fmt rr0}]
Assume first that $T(A)$ is empty.  Then as $A$ has finite nuclear dimension by Proposition \ref{loc fnd}, $A$ is purely infinite by \cite[Theorem 5.4]{Winter:2010eb}, so has real rank zero by the main result of \cite{Zhang:1990aa}.  

Assume next that $T(A)$ is non-empty.  As $A$ has finite nuclear dimension it is in particular exact.  Moreover, if $\mathcal{Z}$ is the Jiang-Su algebra, then as $A$ is simple, unital and has finite nuclear dimension, it is $\mathcal{Z}$-stable by \cite[Corollary 6.3]{Winter:2012qf}.  Using the universal property of the $K_0$-group (see for example \cite[Proposition 3.1.8]{Rordam:2000mz}) it is straightforward to see that the map $p\mapsto \widehat{p}$ from projections in $M_\infty(A)$ to $\text{Aff}(T(A))$ induces a well-defined group homomorphism
$$
\iota_K:K_0(A)\to \text{Aff}(T(A)),\quad [p]-[q]\mapsto \widehat{p}-\widehat{q}.
$$
Using \cite[Theorem 7.2]{Rordam:2004lw}, it suffices to show that $\iota_K$ has uniformly dense image.

Let then $\epsilon>0$, and let $\alpha$ be an element of $\text{Aff}(T(A))$ that we want to approximate uniformly by elements in the image of $\iota_K$.  Replacing $\alpha$ with $\alpha+n\cdot \widehat{1_A}$, we may assume that $\alpha$ is strictly positive.  If there is a projection $q\in M_\infty(A)$ with $\widehat{q}=\alpha$, we are done; assume this does not happen.  

In this case, Lemma \ref{ref lem} gives a projection $p_1\in M_\infty(A)$ such that $\alpha-\widehat{p_1}$ is strictly positive and $\|\alpha-\widehat{p_1}\|\leq \frac{3}{4}\|\alpha\|$.  Set then $\alpha_2:=\alpha-\widehat{p_1}$.  Similarly, if there is a projection $p\in M_\infty(A)$ such that $\alpha_2=\widehat{p}$ then with $q:=p_1\oplus p$, we have $\alpha=\widehat{q}$ by Lemma \ref{iota props} part \eqref{iota sum}, and we have contradicted our assumption that $\alpha$ is not of this form.  Hence $\alpha_2\neq \widehat{p}$ for any $p\in M_\infty(A)$, and so Lemma \ref{ref lem} gives a projection $p_2\in M_\infty(A)$ such that $\alpha_2-\widehat{p_2}$ is strictly positive, and so that $\|\alpha_2-\widehat{p_2}\|\leq \frac{3}{4}\|\alpha_2\|$, which implies that $\|\alpha-\widehat{p_1\oplus p_2}\|\leq (\frac{3}{4})^2\|\alpha\|$.  Continuing in this way, we recursively find a sequence of projections $(p_n)$ in $M_\infty(A)$ such that if $q_n:=p_1\oplus \cdots \oplus p_n$, then $\|\alpha-\widehat{q_n}\|\leq (\frac{3}{4})^n\|\alpha\|$.  Hence $(\widehat{q_n})$ converges uniformly to $\alpha$, and we are done.
\end{proof}

\begin{remark}\label{gen rr0 rem}
We do not know if (weak) complexity rank at most one implies real rank zero without the simplicity and separability assumptions.  This seems an interesting question: for example, the uniform Roe algebra $C^*_u(|\Z|)$ of the integers has complexity rank at most one\footnote{As it contains a proper isometry (for example, the unilateral shift), it is not locally finite dimensional, so has complexity rank exactly one (see \cite[Theorem 2.2]{Li:2017ac} for a more general result along these lines).} by \cite[Example A.9]{Willett:2021te}.  Whether or not $C^*_u(|\Z|)$ has real rank zero is quite an interesting problem: a positive answer would imply the existence of a stably finite $C^*$-algebra with real rank zero but stable rank larger than one (compare the comment at the bottom of page 455 of \cite{Blackadar:2006eq}), while a negative answer would allow one to characterize when uniform Roe algebras have real rank zero.  See the discussion below \cite[Question 3.10]{Li:2017ac} for more details on all this.  

On the other hand, the uniform Roe algebra of $\Z^2$ has complexity rank at most two by \cite[Example A.9]{Willett:2021te} again, and does not have real rank zero by \cite[Theorem 3.1]{Li:2017ac}, so it is certainly not true that (weak) finite complexity implies real rank zero in general.
\end{remark}

\section{Torsion in odd $K$-theory}\label{cr1 sec}

In this section, we show that the $K_1$-group of a $C^*$-algebra with complexity rank at most one is torsion free.  This seems to be of interest in its own right, and is also a key ingredient in our computation of the complexity rank of UCT Kirchberg algebras.

Here is the main theorem of this section.  The result was inspired by a comment of Ian Putnam, who suggested the methods of \cite{Willett:2019aa} could be used to prove something like this.

\begin{theorem}\label{k1tf}
Let $A$ be a unital $C^*$-algebra with complexity rank at most one. Then $K_{1}(A)$ is torsion-free.
\end{theorem}

Before we get into the proof of us, let us use it to show that weak complexity rank and complexity rank are genuinely different.

\begin{corollary}\label{wcr1 is not cr1}
There are $C^*$-algebras with weak complexity rank one that do not have complexity rank one.
\end{corollary}

\begin{proof}
Any unital Kirchberg algebra has weak complexity rank one by Corollary \ref{kirch wcr1}.  A Kirchberg algebra can have any countable abelian group as its $K_1$-group (see \cite[Theorem 3.6]{Rordam:1995aa}, or \cite[Proposition 4.3.3]{Rordam:2002cs}), so by Theorem \ref{k1tf} there are Kirchberg algebras that do not have complexity rank one.
\end{proof}

Throughout this section, if $a\in M_n(A)$, then $a^{\oplus k}$ is the diagonal matrix with all entries $a$ in $M_k(M_n(A))=M_{kn}(A)$.  If $A$ is unital, we write the unit in $M_n(A)$ as $1_n$.  We will rely heavily on ideas from \cite{Willett:2019aa}: we will give precise statements for what we need, but some proofs just refer to that paper.  The methods of proof we use rely on $K$-theory groups based on idempotents and invertibles, not just projections and unitaries as is common in $C^*$-algebra $K$-theory: we recommend \cite[Chapters 5 and 8]{Blackadar:1998yq} as a background reference for this.

The following two lemmas are contained in the proof of \cite[Lemma 2.4]{Willett:2019aa} (see also \cite[Proposition 4.3.2]{Blackadar:1998yq} for the second).

\begin{lemma}\label{approx k el}
For any $c\geq 1$ and $\epsilon>0$ there exists $\delta>0$ with the following property.  Let $A$ be a $C^*$-algebra, $B$ be a $C^*$-subalgebra, and let $e\in M_n(A)$ be an idempotent with $\|e\|\leq c$ and $e\ind M_n(B)$.  Then there is an idempotent $f\in M_n(B)$ with $\|e-f\|<\epsilon$.    \qed
\end{lemma}

\begin{lemma}\label{near k el}
Let $d\geq 1$, and let $A$ be a unital $C^*$-algebra.  If $e,f\in M_n(A)$ are idempotents that satisfy $\|e\|\leq d$, $\|f\|\leq d$, and $\|e-f\|\leq (2d+1)^{-1}$ then the classes $[e]$ and $[f]$ in $K_0(A)$ are the same. \qed
\end{lemma}

Now, assume $c\geq 1$, $\epsilon\in (0,(4c+6)^{-1})$, and let $\delta$ have the property in Lemma \ref{approx k el} for this $c$ and $\epsilon$.  Assume $B$ is a unital $C^*$-subalgebra of a $C^*$-algebra $A$, and that $e\in M_n(A)$ is an idempotent with $\|e\|\leq c$ and $e\ind M_n(B)$.  Then Lemma \ref{approx k el} gives an idempotent $f\in M_n(B)$ with $\|f-e\|< \epsilon$, and so in particular $\|f\|\leq d:=c+1$.  Moreover, if $f'\in M_n(B)$ is another idempotent satisfying $\|f'-e\|< \epsilon$ then $\|f-f'\|<2\epsilon<(2c+3)^{-1}=(2d+1)^{-1}$, so Lemma \ref{near k el} implies that $[f]=[f']$ in $K_0(B)$.  In conclusion, we get a well-defined class in $K_0(B)$ associated to $e$. 

The following is \cite[Definition 2.5]{Willett:2019aa}.

\begin{definition}\label{delt in el}
Assume $c\geq 1$, $\epsilon\in (0,(4c+6)^{-1})$, and let $\delta$ have the property in Lemma \ref{approx k el} for this $c$ and $\epsilon$.  Let $B$ be a unital $C^*$-subalgebra of $A$, and let $e\in M_n(A)$ be an idempotent such that $\|e\|\leq c$, and $e\ind M_n(B)$.  We write $\{e\}_B$ for the class in $K_0(B)$ of any idempotent $f$ in $M_n(B)$ that satisfies $\|e-f\|<\epsilon$ as in the above discussion.  
\end{definition}
 
The following is \cite[Definition 2.6]{Willett:2019aa}.  For the statement of this definition, and the rest of this section, if $E$ is a $C^*$-subalgebra of a unital $C^*$-algebra $A$, we write $\widetilde{E}$ for the $C^*$-subalgebra of $A$ spanned by $E$ and $1_A$.

\begin{definition}\label{lift def}
Let $c\geq 1$, let $\epsilon\in (0,(4c+6)^{-1})$, and let $\delta>0$ satisfy the condition in Lemma \ref{approx k el}.  Let $A$ be a unital $C^*$-algebra, let $C$ and $D$ be $C^*$-subalgebras of $A$. Let $u \in M_{n}(A)$ be an invertible element for some $n$. Then an element $v \in M_{2n}(A)$ is a $(\delta, c, C, D)$-\textit{lift} of $u$ if 
\begin{enumerate}[(i)]
\item \label{lift1} $\|v\|\leq c$ and $\|v^{-1}\|\leq c$;
\item \label{lift2} $v\ind M_{2n}(\widetilde{D})$;
\item \label{lift3} $v\begin{pmatrix} u^{-1} & 0 \\ 0 & u \end{pmatrix}\ind M_{2n}(\widetilde{C})$;
\item \label{lift4} $v\begin{pmatrix} 1_{n} & 0 \\ 0 & 0 \end{pmatrix} v^{-1}\ind M_{2n}(\widetilde{C\cap D})$;
\item \label{lift5} with notation as inDefinition \ref{delt in el}, the $K$-theory class 
$$
\partial_{v}(u) := \Bigg\{v\begin{pmatrix} 1_{n} & 0 \\ 0 & 0 \end{pmatrix} v^{-1}\Bigg\}_{\widetilde{C\cap D}} -\Bigg[\begin{pmatrix} 1_{n} & 0 \\ 0 & 0 \end{pmatrix}\Bigg]\in K_0(\widetilde{C\cap D})
$$ 
is actually in the subgroup $K_0(C\cap D)$.
\end{enumerate}
\end{definition}

We need another definition.

\begin{definition}\label{sigma map}
Let $C$ and $D$ be $C^*$-subalgebras of a $C^*$-algebra $A$, with corresponding inclusion maps $\iota^C:C\to A$ and $\iota^D:D\to A$.  Let $\sigma:K_1(C)\oplus K_1(D)\to K_1(A)$ be the map defined by $\sigma:=\iota^C_*+\iota^D_*$.
\end{definition}

The following result is contained in the proof of \cite[Proposition 2.7]{Willett:2019aa}.

\begin{lemma}\label{exact result}
Let $c\geq 1$, and let $\epsilon\in (0,(4c+6)^{-1})$.  Then there is a $\delta>0$ depending only on $\epsilon$ and $c$, and with the following property.  Let $A$ be a unital $C^*$-algebra and let $u\in M_n(A)$ be an invertible element such that $\|u\|\leq c$ and $\|u^{-1}\|\leq c$.  Let $C$ and $D$ be $C^*$-subalgebras of $A$, and let $v\in M_{2n}(A)$ be a $(\delta,c,C,D)$-lift of $u$ as in Definition \ref{lift def}.  If the $K$-theory class
$$
\partial_v(u):=\Bigg\{v\begin{pmatrix} 1_{n} & 0 \\ 0 & 0 \end{pmatrix} v^{-1}\Bigg\}_{\widetilde{C\cap D}} -\Bigg[\begin{pmatrix} 1_{n} & 0 \\ 0 & 0 \end{pmatrix}\Bigg]
$$
of Definition \ref{lift def} is zero, then the class $[u]\in K_1(A)$ is in the image of the map $\sigma$ from Definition \ref{sigma map}. \qed
\end{lemma}

We need a little more notation before we recall another result from \cite{Willett:2019aa}.  

\begin{definition}\label{vuh}
Let $A$ be a unital $C^*$-algebra, let $h$ be a positive contraction in $A$, and let $u$ be an invertible element of $M_n(A)$.  Define\footnote{Here we conflate a contraction $h\in A$ with the corresponding diagonal matrix $h \otimes 1_{n} \in M_{n}(A)$} $a:=h+(1-h)u\in M_n(A)$ and $b:=h+(1-h)u^{-1}\in M_n(A)$.  Define 
$$
v(u,h):=\begin{pmatrix} 1 & a \\
0 & 1
\end{pmatrix}
\begin{pmatrix}
1 &  0\\
-b & 1
\end{pmatrix}
\begin{pmatrix}
1 & a\\
0 & 1
\end{pmatrix}
\begin{pmatrix}
0 & -1\\
1 & 0
\end{pmatrix}\in M_{2n}(A).
$$
\end{definition}

The following result is contained in the proof of \cite[Proposition 3.6]{Willett:2019aa}.

\begin{lemma}\label{lifts exist}
For any $\delta>0$ and $n\in \N$ there exists $\gamma>0$ with the following property.  Let $A$ be a unital $C^*$-algebra and $u\in M_n(A)$ be a unitary.   Let $X\subseteq A$ be a (possibly infinite) subset of $A$ containing the matrix entries of $u$.  

Then if $(h,C,D)$ is a triple satisfying the conditions in Lemma \ref{d equiv} with respect to $X$ and any $\epsilon\in (0,\gamma]$, then $v(u,h)$ is a $(\delta,8,C,D)$ lift of $u$. \qed
\end{lemma}

For $k\in \N$, let $s_k\in M_{2k}(\C)$ be the (unitary) permutation matrix determined by 
$$
(z_1,z_2,...,z_k,z_{k+1},...,z_{2k})\mapsto (z_1,z_3,....,z_{2k-1},z_2,z_4,...,z_{2k}).\
$$
For any unital $C^*$-algebra $A$ and any $n\in \N$, we abuse notation by identifying $s_k$ with the element $s_k\otimes 1_{M_n(A)}$ of $M_{kn}(A)=M_k(\C)\otimes M_n(A)$.  

The following fact is closely related to \cite[Lemma 4.2]{Willett:2019aa}\footnote{The statement of \cite[Lemma 4.2]{Willett:2019aa} claims that the element we have called $s_k$ is self-inverse, which is clearly wrong.  However, this does not significantly affect that lemma, having replaced $s_k(v^{\oplus k})s_k$ with $s_k(v^{\oplus k})s_k^*$ as appropriate.}.  The proof consists in direct checks that we leave to the reader.

\begin{lemma}\label{mult lem}
Let $c\geq 1$, and $\epsilon\in (0,(4c+6)^{-1})$.   Let $A$ be a unital $C^*$-algebra, let $u\in M_n(A)$ be unitary, and let $v \in M_{2n}(A)$ be a $(\delta, c, C, D)$-lift of $u$, where $\delta$, $C$, and $D$ satisfy the conditions in Definition \ref{lift def}.  Then the following hold:
\begin{enumerate}[(i)]
\item \label{welldef part} For any $k\in \N$, $s_k(v^{\oplus k} )s_k^{*}$ is a $(\delta, c, C, D)$-lift of $u^{\oplus k}$. 
\item \label{times k part} The $K$-theory classes $\partial_{s_k(v^{\oplus k})s_k^*}(u^{\oplus k})$ and $k\cdot \partial_v(u)$ are equal in $K_0(C\cap D)$. 
\item \label{vuh part} If $v=v(u,h)$ is given by the formula in Definition \ref{vuh}, then $s_k(v^{\oplus k})s_k^*=v(u^{\oplus k},h)$. \qed
\end{enumerate}
\end{lemma}

\begin{proof}[Proof of Theorem \ref{k1tf}]
Let $\kappa \in K_{1}(A)$ be such that $n\cdot\kappa = 0$ for some $n\in \N$; our goal is to show that $n=0$ or $\kappa=0$.  Let $w \in M_{m}(A)$ be a unitary element such that $[w] = \kappa$.  As $[w^{\oplus n}] = n \cdot \kappa = 0$ we have that $w^{\oplus n} \oplus 1_{r}$ is homotopic through unitaries to $1_{nm+r}$ for some $r$.  Letting $s=nm+nr$, we have that $(w\oplus 1_r)^{\oplus n}$ is homotopic via a rotation homotopy to $w^{\oplus n}\oplus 1_r\oplus 1_{(n-1)r}$, and therefore is homotopic to $1_s$.   Define $u:= w \oplus 1_{r}$, so $[u^{\oplus n}] = n \cdot \kappa = 0$ and $u^{\oplus n}$ is homotopic to $1_{s}$.  Let $(u_{t})_{t \in [0,1]}$ denote a path of unitary elements in $M_s(A)$ with $u_0=u^{\oplus n}$ and $u_1=1_s$.  

Let $c=8$, and let $\epsilon=(12c+18)^{-1}$.  Let $\delta>0$ have the property in Lemma \ref{exact result} with respect to this $\epsilon$ and $c$.  Let $\gamma_{m+r}$ (respectively, $\gamma_{n(m+r)}$) be as in Lemma \ref{lifts exist} with respect to $\delta$ and the integer $m+r$ (respectively, $m(n+r)$), and let $\gamma:=\min \{\gamma_{m+r},\gamma_{n(m+r)}\}$.  Choose a finite partition $0 = t_{0} < ... < t_{k} = 1$ of $[0,1]$ such that for any $t \in [t_i,t_{i+1}]$ we have $\| u_{t} - u_{t_{i}} \| < \gamma/2$.  For each $i\in \{0,...,k\}$, let $X_t\subseteq A$ be the finite subset consisting of all matrix entries of $u_{t}$.  Let $X:=\bigcup_{i=1}^k X_{t_i}$.  Let $(h,C,D)$ be a triple satisfying the conditions in Proposition \ref{d equiv} with respect to the finite set $X$ and error parameter $\gamma/2$.  Note that for any $t$ and any $x\in X_t$ there are $i$ and $x_{t_i}$ such that $\|x-x_{t_i}\|<\gamma/2$.  It follows that $(h,C,D)$ satisfies the conditions in Lemma \ref{d equiv} with respect to the (possibly infinite) set $\bigcup_{t} X_t$ and the error parameter $\gamma$.

At this point Lemma \ref{lifts exist} gives us that (with notation as in Definition \ref{vuh}) $v_t:=v(u_t,h)$ is a $(\delta,8,C,D)$-lift for $u_t$ for all $t$.  The choice of $\delta$ then gives us elements 
$$
\partial_{v_t}(u_t):= \Bigg\{v_t\begin{pmatrix} 1_{s} & 0 \\ 0 & 0 \end{pmatrix} v_t^{-1}\Bigg\}_{\widetilde{C\cap D}} -\bigg[\begin{pmatrix} 1_{s} & 0 \\ 0 & 0 \end{pmatrix}\bigg] \in K_0(C\cap D)
$$ 
for all $t$.  We claim that 
\begin{equation}\label{hom inv}
\partial_{v_1}(u_1)=\partial_{v_0}(u_0).
\end{equation}
The path 
$$
[0,1]\to M_{2s}(A),\quad t\mapsto v_t\begin{pmatrix} 1_{s} & 0 \\ 0 & 0 \end{pmatrix} v_t^{-1}
$$ 
is continuous, whence there exists a finite partition $0=t_0<\cdots <t_l=1$ such that 
$$
\Bigg\|v_{t_{j+1}}\begin{pmatrix} 1_{s} & 0 \\ 0 & 0 \end{pmatrix} v_{t_{j+1}}^{-1}-v_{t_j}\begin{pmatrix} 1_{s} & 0 \\ 0 & 0 \end{pmatrix} v_{t_j}^{-1}\Bigg\|<\epsilon
$$
for all $j\in \{0,...,l-1\}$.  Hence if $f_j$ and $f_{j+1}$ are idempotents in $M_{2s}(\widetilde{C\cap D})$ satisfying 
$$
\Bigg\|f_j-v_{t_j}\begin{pmatrix} 1_{s} & 0 \\ 0 & 0 \end{pmatrix} v_{t_j}^{-1}\Bigg\|<\epsilon
$$
then $\|f_j-f_{j+1}\|<3\epsilon= (4c+6)^{-1}$.  Hence $[f_j]=[f_{j+1}]$ in $K_0(\widetilde{C\cap D})$ for all $j$ by Lemma \ref{near k el}, whence the claim.

As $u_1=1_s$ and the formula from Definition \ref{vuh} implies that $v_1=\begin{pmatrix} 1_s & 0 \\ 0 & 1_s\end{pmatrix}$, whence $\partial_{v_1}(u_1)=0$ by the formula from Definition \ref{lift def}, part \eqref{lift5}.  Hence by the claim from line \eqref{hom inv} that we just established, 
\begin{equation}\label{puv0}
\partial_{v_0}(u_0)=0.
\end{equation}
Let $v=v(u,h)$, which is a $(\delta,8,C,D)$ lift of $u$ by our choice of $\gamma$, and Lemma \ref{mult lem} part \eqref{welldef part} implies $s_n(v^{\oplus n})s_n^*$ is a $(\delta,8,C,D)$ lift of $u^{\oplus n}$, whence the element $\partial_{s_n(v^{\oplus n})s_n^*}(u^{\oplus n})$ of Definition \ref{lift def} makes sense.  On the other hand, part \eqref{vuh part} of Lemma \ref{mult lem} implies that $s_n(v^{\oplus n})s_n^*=v(u^{\oplus n},h)=v_0$, and so the classes $\partial_{s_n(v^{\oplus n})s_n^*}(u^{\oplus n})$ and $\partial_{v_0}(u_0)$ are equal.  Hence $\partial_{s_n(v^{\oplus n})s_n^*}(u^{\oplus n})=0$ by line \eqref{puv0}.  Part \eqref{times k part} of Lemma \ref{mult lem} implies that $n\cdot \partial_v(u)=\partial_{s_n(v^{\oplus n})s_n^*}(u^{\oplus n})$ so we get
\begin{equation}\label{torsion}
n\cdot \partial_v(u)=0.
\end{equation}

Now, as $C\cap D$ is finite-dimensional, $K_0(C\cap D)$ is torsion-free, so line \eqref{torsion} forces $n=0$ or $\partial_v(u)=0$.  If $n=0$ we are done, so assume $\partial_v(u)=0$.  From Lemma \ref{exact result}, we thus have that $[u]$ is in the range of $\sigma$.  However, the domain of $\sigma$ is $K_1(C)\oplus K_1(D)$, which is zero as $C$ and $D$ are finite-dimensional.  Hence $[u]=0$; as $u=w\oplus 1_r$, this implies that $[w]=0$ too, and we are done.
\end{proof}

\section{Kirchberg algebras}\label{kirch sec}

Our goal in this section is to study the complexity rank of Kirchberg algebras.  Recall that a $C^*$-algebra is a \emph{Kirchberg algebra} if it is simple, separable, nuclear and purely infinite.  Our theorems will only apply to unital Kirchberg algebras, but non-unital Kirchberg algebras will play an important role in the proof.

The following theorem is our goal in this section.

\begin{theorem}\label{kirch the}
Let $A$ be a unital UCT Kirchberg algebra.  Then $A$ has complexity rank one or two.  Moreover, it has complexity rank two if and only if $K_1(A)$ contains non-trivial torsion elements.
\end{theorem}

There is quite a striking contrast here with the theory of nuclear dimension (and with weak complexity rank).  Indeed, all Kirchberg algebras have nuclear dimension one\footnote{This holds regardless of the UCT: it was established earlier under a UCT assumption in \cite{Ruiz:2014aa}.} by \cite[Theorem G]{Bosa:2014zr}; as a consequence of this and real rank zero, all Kirchberg algebras have \emph{weak} complexity rank one as recorded in Corollary \ref{kirch wcr1} above.  As already noted in the introduction, proving Theorem \ref{kirch the} (or even an a priori much weaker statement, such as that a Kirchberg algebra with zero $K$-theory has finite complexity) \emph{without the UCT assumption} would imply the UCT for all nuclear $C^*$-algebras.

The proof of Theorem \ref{kirch the} will make repeated use of (part of) the Kirchberg-Phillips classification theorem  \cite{Phillips-documenta}; see also the exposition in \cite[Chapter 8]{Rordam:2002cs} and the recent approach in \cite{Gabe:2019ws}.

\subsection{The rank one case (after Enders)}

In this subsection, we adapt ideas of Enders \cite{Enders:2015aa} to establish the following theorem.

\begin{theorem}\label{kirch the 1}
Let $B$ be a unital UCT Kirchberg algebra with torsion free $K_1$ group.  Then $B$ has complexity rank one.
\end{theorem}

Throughout this subsection we will be dealing with large matrices, so adopt some shorthand notation for convenience.  Let $e_{ij}$ denote the matrix units in $M_n(\C)$, and for $j\in \{-(n-1),...,0,1,...,n-1\}$, write $d_j$ for the matrix which has ones on the $j^\text{th}$ subdiagonal and is zero elsewhere, i.e.\ $d_j:=\sum_{i=1}^{n-j}e_{(i+j)i}$.  Note for example that $d_0$ is the identity, and that $d_{-1}=e_{12}+e_{23}+\cdots +e_{(n-1)n}$ is the matrix with ones on the first \emph{super}diagonal and zeros elsewhere.  For a $C^*$-algebra $B$ with multiplier algebra $M(B)$, we will identify $M_n(M(B))$ with $M_n(\C)\otimes M(B)$; for example, if $u\in M(B)$ is unitary, then we will write things like 
$$
d_{1} \otimes 1 + d_{-(n-1)} \otimes u^{n} \in M_{n}(\mathbb{C}) \otimes M(B).
$$
for the matrix 
\begin{equation}\label{inu matrix}
\begin{pmatrix} 0 & 0 & \cdots & 0 & 0 & u^n \\ 1 & 0 & \cdots & 0 & 0 & 0 \\  0 & 1 & \cdots & 0 & 0 & 0  \\ \vdots & \vdots & \ddots & \vdots & \vdots & \vdots \\  0 & 0 & \cdots & 1 & 0 & 0 \\  0 & 0 & \cdots & 0 & 1 &  0\end{pmatrix} \in M_n(M(B)).
\end{equation}
For a $C^*$-algebra $B$ and $b_1,...,b_n\in B$, we will  write $\text{diag}(b_1,...,b_n)$ for the diagonal matrix in $M_n(B)$ with entries $b_1,...,b_n$, i.e.\ for 
$$
\begin{pmatrix} b_1 & 0 & \cdots  & 0 \\ 0 & b_2 & \cdots  & 0 \\ \vdots & \vdots & \ddots & \vdots \\ 0 & 0 & \cdots &  b_n\end{pmatrix}\in M_n(B)
$$
We need a definition: the following is \cite[Definition 1.1]{Enders:2015aa}.  

\begin{definition}\label{iotan}
Let $A$ be a $C^*$-algebra equipped with an action $\alpha$ of $\Z$, and let $n\in \N$.  Let $\iota_n$ be the $*$-homomorphism
$$
\iota_n:A\rtimes \Z\to M_n(A\rtimes \Z)
$$
determined by the formulas 
$$
\iota_n(a):=\text{diag}(\alpha^{-1}(a), \alpha^{-2}(a),...,\alpha^{-n}(a))
$$
for $a\in A$ and
$$
\iota_n(u):=d_{1} \otimes 1 + d_{-(n-1)} \otimes u^{n}
$$ 
for $u\in M(A\rtimes \Z)$ the canonical unitary implementing the $\Z$ action\footnote{Line \eqref{inu matrix} above with $B=A\rtimes \Z$ writes out $\iota_n(u)$ as a matrix.}.
\end{definition}

The key technical result is as follows: although somewhat different from the conclusions of Ender's arguments, it follows the same basic strategy.

\begin{lemma}\label{enders lem}
Let $A$ be an AF $C^*$-algebra equipped with a $\Z$-action.  Let $X$ be a finite subset of $A\rtimes \Z$, and assume there exists a projection $p\in A$ such that $px=xp=x$ for all $x\in X$.  Let $\epsilon>0$.

Then there exists $N\in \N$ such that for all $n\geq N$, if $\iota_n:A\rtimes \Z\to M_n(A\rtimes\Z)$ is as in Definition \ref{iotan}, and $q:=\iota_n(p)$, then there exists a positive contraction $h\in q(M_n(A\rtimes \Z))q$ and AF $C^*$-subalgebras  $C$ and $D$ of $q(M_n(A\rtimes \Z))q$ with the following properties:
\begin{enumerate}[(i)]
\item \label{ki com} $\|[h,\iota_n(x)]\|<\epsilon$ for all $x\in X$;
\item \label{ki in} $h\iota_n(x)\in_{\epsilon} C$, $(1-h)\iota_n(x)\in_{\epsilon} D$, and $(1-h)h\iota_n(x)\in_{\epsilon} C\cap D$ for all $x\in X$;
\item \label{ki af} $E:=C\cap D$ is an AF algebra;
\item \label{ki mult} $h$ multiplies $E$ into itself.  
\end{enumerate}
\end{lemma}

\begin{proof}
Define a unitary $v \in M_n(\C)\otimes M(A\rtimes \Z)=M(M_n(A\rtimes \Z))$ by 
\begin{align*}
v:= d_{1} \otimes u^{-1} + d_{-(n-1)} \otimes u^{n-1}=\begin{pmatrix}
0 & \dots & \dots &  0 & u^{n-1} \\
u^{-1} & \ddots & & & 0 \\
0 & \ddots & \ddots &  & \vdots \\
\vdots & \ddots & \ddots & \ddots & \vdots \\   
0 & \dots & 0 & u^{-1} &  0\\ 
\end{pmatrix}.
\end{align*} 
Write $n = 2m$ if $n$ is even and $n = 2m + 1$ if $n$ is odd, and note that 
\begin{equation}\label{vm form}
v^{m} = d_{m} \otimes u^{-m} + d_{-(n-m)} \otimes u^{n-m}.
\end{equation}
Let $j_n: M_{n}(A)  \hookrightarrow M_{n}(A \rtimes_{\alpha} \mathbb{Z})$ denote the canonical inclusion and define two $*$-homomorphisms 
$$
\Lambda_{n}^{0}, \Lambda_{n}^{1}: M_{n}(A) \rightarrow M_{n}(A \rtimes_{\alpha}\mathbb{Z}),\quad \Lambda_{n}^{0}:= j_{n},\quad \Lambda_{n}^{1} := Ad_{v^{m}} \circ j_{n}.
$$
Let us compute the image of $\Lambda_{n}^{1}$ more concretely.  Write elements in $M_n(A)$ in the form
\begin{equation}\label{mat split}
\begin{pmatrix}
a & b \\
c & d \\       
\end{pmatrix} \in M_{(n-m)+m}(A)
\end{equation}
where writing $n$ as the sum of $n-m$ and $m$ in the subscript on the right records the sizes of the blocks.  Using line \eqref{vm form}, one  then computes that $\Lambda_n^1$ acts via sending the matrix in line \eqref{mat split} to the element  
\begin{equation}\label{ln1 act}
 \begin{pmatrix}
\alpha^{n-m}(d) & 0 \\
0 & \alpha^{-m}(a) \\       
\end{pmatrix} + \begin{pmatrix}
0 & \alpha^{n-m}(c) \\
0 & 0 \\       
\end{pmatrix}\cdot u^{n} + \begin{pmatrix}
0 & 0 \\
\alpha^{-m}(b) & 0 \\       
\end{pmatrix}\cdot u^{-n}
\end{equation}
in $M_{m+(n-m)}(A\rtimes \Z)$ (note the switch from ``$(n-m)+m$'' to ``$m+(n-m)$'').   Define also 
$$
q:=\iota_n(p)=\text{diag}(\alpha^{-1}(p),...,\alpha^{-n}(p)).
$$
One checks directly that $q$ multiplies $\Lambda_n^0(M_n(A))$ into itself, while the fact that $q$ multiplies $\Lambda_n^1(M_n(A))$ into itself follows from the formula in line \eqref{ln1 act} above.  Hence $C:=q(\Lambda_n^0(M_n(A)))q$ and $D:=q(\Lambda_n^1(M_n(A)))q$ are well-defined AF subalgebras of $M_n(A\rtimes \Z)$.  Note moreover that with respect to the decomposition in line \eqref{mat split}, the intersection of $C$ and $D$ can be concretely described as the set 
$$
E:=\Bigg\{ q\begin{pmatrix} a & 0 \\ 0 & d \end{pmatrix}q ~\Bigg|~a\in M_m(A),~d\in M_{n-m}(A)\Bigg\}
$$  
and is in particular also an AF algebra.  

For $1\leq i \leq n$, define scalars $h_{i}\in [0,1]$ by  
\begin{align}\label{zh 1}
	 h_{i} := \left\{
\begin{array}{ll}
\;\;\;\;\;\;\;0 & \;\;\;1 \leq i \leq \floor*{\frac{n}{6}}  \\
\dfrac{i - \floor*{\frac{n}{6}}}{\floor*{\frac{2n}{6}} - \floor*{\frac{n}{6}}} & \floor*{\frac{n}{6}}\leq i \leq \floor*{\frac{2n}{6}} \\
\;\;\;\;\;\;1 & \floor*{\frac{2n}{6}}\leq i \leq \floor*{\frac{4n}{6}} \\
\dfrac{\floor*{\frac{5n}{6}} - i}{\floor*{\frac{5n}{6}} - \floor*{\frac{4n}{6}}} & \floor*{\frac{4n}{6}}\leq i \leq \floor*{\frac{5n}{6}} \\
\;\;\;\;\;\;\;0 & \floor*{\frac{5n}{6}}\leq i \leq n \\
\end{array} 
\right. 
\end{align} and let $h\in M_n(A)$ be defined by 
\begin{align*}
	h :=\text{diag}(h_1,...,h_n)q=\text{diag}(h_1\alpha^{-1}(p),...,h_n\alpha^{-n}(p)).
\end{align*} 
Note that $h$ multiplies $C$ and $D$ into themselves, whence it also multiplies $E$ into itself.

We claim now that for $n$ suitably large, $C$, $D$, $E$, and $h$ have the properties claimed in the statement of the lemma.  We have already observed properties \eqref{ki af} and \eqref{ki mult}, so it remains to check properties \eqref{ki com} and \eqref{ki in}.

Let us look first at property \eqref{ki com}.  As $q\iota_n(x)=\iota_n(px)=\iota_n(x)=\iota_n(xp)=\iota_n(x)q$ for all $x\in X$, and as $q$ commutes with 
\begin{equation}\label{h0 def}
h^{(0)} := \text{diag}(h_1,...,h_n)
\end{equation}
it suffices to show that 
\begin{equation}\label{h0 com}
[h^{(0)},\iota_n(x)]\to 0 \quad \text{as}\quad n\to\infty.
\end{equation}
To show this, it suffices to show that for any contraction $a\in A$ and any $k\in \N\cup\{0\}$ we have that $[h^{(0)},\iota_n(a\cdot u^k)]\to 0$ as $n\to\infty$.

Fix then $y= a\cdot u^{k}$ and compute $h^{(0)}\cdot \iota_{n}(y):$
\begin{align}\label{h0inx}
h^{(0)}\cdot\iota_{n}(y) &= h^{(0)}\cdot\iota_{n}(a \cdot u^{k}) \nonumber \\ &= h^{(0)}\cdot\iota_{n}(a)\iota_{n}(u^{k}) \nonumber \\ 
&= h^{(0)}\text{diag}(\alpha^{-1}(a),...,\alpha^{-n}(a)) \big(d_{k}\otimes 1_{n} + d_{-n+k}\otimes u^{n} \big) \nonumber  \\
&= \begin{pmatrix}
0 & 0 \\
g_{1} & 0 \\       
\end{pmatrix} + \begin{pmatrix}
0 & g_{2} \\
0 & 0 \\       
\end{pmatrix} \cdot u^{n}
\end{align}
where $g_{1}$ is the $(n-k) \times (n-k)$ matrix 
\begin{equation}\label{g1}
g_{1} := \text{diag} (h_{k+1}\alpha^{-(k+1)}(a), h_{k+2}\alpha^{-(k+2)}(a),..., h_{n}\alpha^{-n}(a))
\end{equation}
and $g_{2}$ is the $k\times k$ matrix  
\begin{equation}\label{g2}
g_{2} = \text{diag} (h_{1}\alpha^{-1}(a),..., h_{k}\alpha^{-k}(a)).
\end{equation}
If we choose $n$ large enough so that $k << \floor*{\frac{n}{6}}$, then $\begin{psmallmatrix}0 & g_{2}\\0 & 0\end{psmallmatrix}\cdot u^{n} = 0$, since $h_{i} = 0$ for $1 \leq i \leq \floor*{\frac{n}{6}}$.  Hence for large enough $n$ \begin{align} \label{lam 0}
h^{(0)}\cdot\iota_{n}(y) = \begin{pmatrix}
0 & 0 \\
g_{1} & 0 \\       
\end{pmatrix}.
\end{align}
Computing $\iota_{n}(y)\cdot h^{(0)}$ is similar:  if again, $n$ is large enough so that $ k << \floor*{\frac{n}{6}}$, we have
$$
\iota_{n}(y)\cdot h^{(0)} = \begin{pmatrix}
0 & 0 \\
e & 0 \\       
\end{pmatrix},$$
where $e$ is the $(n-k) \times (n-k)$ matrix defined by 
\begin{align*}
e: = \text{diag}( h_{1}\alpha^{-(k+1)}(a), h_{2}\alpha^{-(k+2)}(a),..., h_{n-k}\alpha^{-n}(a)).
\end{align*} 
At this point, if we let $M_{n} := \max\{ \floor*{\frac{2n}{6}} - \floor*{\frac{n}{6}}, \floor*{\frac{5n}{6}} - \floor*{\frac{4n}{6}}  \}$, then we compute that for all large enough $n$, 
\begin{align*}
\| h^{(0)}\cdot \iota_{n}(y) - \iota_{n}(y)\cdot h^{(0)} \| &= \bigg\| \begin{pmatrix}
0 & 0 \\
g_{1} & 0 \\       
\end{pmatrix} - \begin{pmatrix}
0 & 0 \\
e & 0 \\       
\end{pmatrix} \bigg\| \\
& = \| g_{1} - e \| \\
& = \max\limits_{k+1 \leq i \leq n} \big\| h_{i-k}\alpha^{-i}(a) - h_{i}\alpha^{-i}(a)\big\| \\
& \leq \max\limits_{k+1 \leq i \leq n} | h_{i-k} - h_{i}| \big\|\alpha^{-i}(a)\big\| \\
& \leq \dfrac{k}{ M_n } \rightarrow 0 \text{ as } n \rightarrow \infty,
\end{align*}
which completes the proof of condition \eqref{ki com}.  

We now look at condition \eqref{ki in}.  Define $C^{(0)}:=\Lambda_n^0(M_n(A))$, $D^{(0)}:=\Lambda_n^1(M_n(A))$ and $E^{(0)}:=C^{(0)}\cap D^{(0)}$.  Then with $h^{(0)}$ as in \eqref{h0 def}, it suffices to show that for $n$ suitably large and any $x\in X$, $h^{(0)}\iota_n(x)\in_{\epsilon} C^{(0)}$, $(1-h^{(0)})\iota_n(x)\in_{\epsilon} D^{(0)}$, and that $h^{(0)}(1-h^{(0)})\iota_n(x)\in_{\epsilon} E^{(0)}$.  Similarly to the above, it suffices to show that if $a\in A$ is a contraction and if $k\in \N\cup\{0\}$, then for $y:=a\cdot u^k$ we have $h^{(0)}\iota_n(y)\in C^{(0)}$, $(1-h^{(0)})\iota_n(y)\in D^{(0)}$, and that $h^{(0)}(1-h^{(0)})\iota_n(y)\in E^{(0)}$.

First, note that it follows from the computation of $h^{(0)}\cdot\iota_{n}(y)$ in line \eqref{lam 0} that $h^{(0)}\cdot \iota_{n}(y) \in C^{(0)}$ for large enough $n$.  To see that $(1-h^{(0)})\cdot\iota_{n}(y) \in D^{(0)}$, analogously to lines \eqref{h0inx} and \eqref{g1} and \eqref{g2} above, we compute that 
\begin{equation}\label{iot 2}
(1-h^{(0)})\cdot\iota_{n}(y)=\begin{pmatrix}
0 & 0 \\
f_{1} & 0 \\       
\end{pmatrix} + \begin{pmatrix}
0 & f_{2} \\
0 & 0 \\       
\end{pmatrix} \cdot u^{n}
\end{equation}
where $f_{1}$ is the $(n-k) \times (n-k)$ matrix given by 
$$
f_{1} = \text{diag}((1-h_{k+1})\alpha^{-(k+1)}(a), (1-h_{k+2})\alpha^{-(k+2)}(a) ,..., (1-h_{n})\alpha^{-n}(a)) 
$$
and as long as $n$ is chosen large enough so that $k << \floor*{\frac{n}{6}}$, $f_{2}$ is the $k\times k$ matrix given by 
$$
f_2:=\text{diag}(\alpha^{-(k+1)}(a),...,\alpha^{-n}(a)).
$$
Note that $(1-h_{i})=0$ for $\floor*{\frac{2n}{6}}+1 \leq i \leq \floor*{\frac{4n}{6}}$.  Thus the matrix $\begin{psmallmatrix}0 & 0\\f_{1} & 0\end{psmallmatrix}$ can be written as the following sum 
\begin{align*}
& \Big(d_{k}\cdot \text{diag}\big((1-h_{k+1})\alpha^{-(k+1)}(a), ... , (1-h_{\floor*{\frac{n}{3}}})\alpha^{-\floor*{\frac{n}{3}}}(a), 0, ..., 0 \big)   \Big) \\
	&+ \Big(d_{k}\cdot \text{diag}\big(0, ..., 0, (1-h_{\floor*{\frac{2n}{3}}})\alpha^{-\floor*{\frac{2n}{3}}}(a), ... , (1-h_{n})\alpha^{-n}(a) \big)   \Big) \\
	&=\begin{pmatrix}f_{3} & 0 \\ 0 & f_{4}\end{pmatrix},
\end{align*} 
where $f_{3}$ is an $m \times m$ matrix built from the entries of the first summand in the middle line above, and $f_{4}$ is an $(m+1)\times(m+1)$ matrix built from the entries in the second summand in the middle above.  Comparing this with line \eqref{iot 2} above, we see that $(1-h^{(0)})\cdot\iota_{n}(y) \in D^{(0)}$.

Finally, we consider $E$.  We already have that \begin{align*}
	(1-h^{(0)})\cdot\iota_{n}(y) = \begin{pmatrix}
	f_{3} & 0 \\
	0 & f_{4} \\       
	\end{pmatrix} +  \begin{pmatrix}
	0 & f_{2} \\
	0 & 0 \\       
	\end{pmatrix} \cdot u^{n}.
\end{align*}
Multiplying by $h^{(0)}$ on the left will make the second term zero as $h_{i} = 0$ for $1 \leq i \leq \floor*{\frac{n}{6}}$ and $n$ has been chosen so that $k << \floor*{\frac{n}{6}}$.  Thus $h^{(0)}(1-h^{(0)})\cdot\iota_{n}(y) \in E$ and we are done.
\end{proof}

\begin{corollary}\label{end cor}
Let $A$ be an AF $C^*$-algebra equipped with a $\Z$-action.  Let $X$ be a finite subset of $A\rtimes \Z$, and assume there exists a projection $p\in A$ such that $px=xp=x$ for all $x\in X$.  Let $\epsilon>0$.  For each $n\in \N$, define
$$
\phi_n:M_{n}(A\rtimes \Z)\oplus M_{n+1}(A\rtimes \Z)\to M_{2n+1}(A\rtimes \Z),\quad (a,b)\mapsto \begin{pmatrix} a & 0 \\ 0 & b \end{pmatrix},
$$
let $\omega:M_{n}(A\rtimes \Z)\to M_{n}(A\rtimes \Z)$ be any $*$-isomorphism, let $\iota_n$ be as in Definition \ref{iotan}, and define
$$
\kappa_n:=\phi_n\circ ((\omega\circ \iota_{n})\oplus \iota_{n+1}) : A\rtimes \Z\to M_{2n+1}(A\rtimes \Z).
$$
Then there exists $N\in\N$ such that for all $n\geq N$, if $q:=\kappa_n(p)$ there is a positive contraction $h\in q(M_{2n+1}(A\rtimes\Z))q$ and AF $C^*$-subalgebras $C$ and $D$ of $q(M_{2n+1}(A\rtimes \Z))q$ with the following properties:
\begin{enumerate}[(i)]
\item $\|[h,\kappa_n(x)]\|<\epsilon$ for all $x\in X$;
\item $h\kappa_n(x)\in_{\epsilon} C$, $(1-h)\kappa_n(x)\in_{\epsilon} D$, and $(1-h)h\kappa_n(x)\in_{\epsilon} C\cap D$ for all $x\in X$;
\item $E:=C\cap D$ is an AF algebra;
\item $h$ multiplies $E$ into itself.
\end{enumerate}
\end{corollary}

\begin{proof}
Let $N$ be large enough so that the conclusion of Lemma \ref{enders lem} holds for all $n\geq N$ with respect to the given $X$, $\epsilon$ and $p$.   Fix $n\geq N$.  Let $C_n$, $D_n$ be subalgebras of $\iota_n(p)(M_n(A\rtimes \Z))\iota_n(p)$ and $h_n$ a positive contraction in $\iota_n(p)(M_n(A\rtimes \Z))\iota_n(p)$ with the properties in Lemma \ref{enders lem} and similarly for $C_{n+1}$, $D_{n+1}$ and $h_{n+1}$ with respect to $\iota_{n+1}(p)(M_{n+1}(A\rtimes \Z))\iota_{n+1}(p)$.

Define  
$$
C:=\phi_n (\omega(C_n)\oplus C_{n+1}),\quad D:=\phi_n (\omega(D_n)\oplus D_{n+1})$$
and 
$$
h:=\phi_n (\omega(h_n)\oplus h_{n+1}).
$$
Direct checks show these elements have the right properties.  
\end{proof}

Enders computes the effect of $\iota_n$ on $K$-theory: the following result is a special case of \cite[Proposition 3.2]{Enders:2015aa}.

\begin{lemma}\label{enders k}
Let $A$ be a $C^*$-algebra with $K_1(A)=0$, and equipped with an action of $\Z$.  Let $\iota_n:A\rtimes \Z\to M_n(A\rtimes \Z)$ be as in Definition \ref{iotan} and let $i_n:A\rtimes \Z\to M_n(A\rtimes \Z)$ be the standard top-left-corner inclusion.  Then as maps on $K$-theory, $(\iota_n)_*=n\cdot (i_n)_*$. \qed
\end{lemma}

For the next step, we need to use part of the Kirchberg-Phillips classification theorem.  For the reader's convenience, we state the versions of the Kirchberg-Phillips theorem we will use, and how to deduce them from the literature.  

\begin{theorem}[Kirchberg-Phillips]\label{kp stab the}
\begin{enumerate}[(i)]
\item \label{kp stab exist kk} Let $A$ and $B$ be stable Kirchberg algebras.  Then for any invertible element $x$ of $KK(A,B)$, there exists a $*$-isomorphism $\phi:A\to B$ such that $[\phi]=x$. 
\item \label{kp stab exist} Let $A$ and $B$ be stable UCT Kirchberg algebras.  Then for any (graded) isomorphism $\alpha:K_*(A)\to K_*(B)$, there exists a $*$-isomorphism $\phi:A\to B$ that induces $\alpha$.  
\item \label{kp stab unique} Let $A$ and $B$ be stable UCT Kirchberg algebras, and let $\phi,\psi:A\to B$ be $*$-isomorphisms that induce the same class in $KK(A,B)$.  Then there is a sequence of unitaries $(u_n)$ in the multiplier algebra of $B$ such that $u_n\phi(a)u_n^*\to \psi(a)$ as $n\to\infty$ for all $a\in A$. 
\end{enumerate}
\end{theorem}

\begin{proof}
Parts \eqref{kp stab exist kk} and \eqref{kp stab exist} are exactly \cite[Theorem 8.4.1, (i) and (ii)]{Rordam:2002cs}.   Part \eqref{kp stab unique} can be deduced from \cite[Theorem 8.2.1 (ii)]{Rordam:2002cs}\footnote{
The references we give here are to a readable textbook exposition that explains the ideas, but does not quite contain complete proofs.  For references with proofs that one can deduce the results from, see \cite[Theorem 4.2.1]{Phillips-documenta} or \cite[Theorem C]{Gabe:2019ws} for part \eqref{kp stab exist kk},  \cite[Theorem 4.2.4]{Phillips-documenta} or \cite[Theorem D]{Gabe:2019ws} for part \eqref{kp stab exist}, and \cite[Theorem 4.1.3]{Phillips-documenta} or \cite[Theorem C]{Gabe:2019ws} for part \eqref{kp stab unique}.}.
\end{proof}

The next result again follows Enders' work: the proof proceeds along similar lines to \cite[Proof of Theorem 4.1]{Enders:2015aa}

\begin{corollary}\label{enders and kp}
Let $A$ be an AF algebra equipped with an action of $\Z$ so that the associated crossed product $A\rtimes \Z$ is a Kirchberg algebra\footnote{This assumption forces $A$ to be non-unital, whence $A\rtimes \Z$ is stable by Zhang's dichotomy: see \cite[Proposition 4.1.3]{Rordam:2002cs}, or \cite[Theorem 1.2]{Zhang:1992vc}  for the original reference.}.  Let $X$ be a finite subset of $A\rtimes \Z$, and assume there exists a projection $p\in A$ such that $px=xp=x$ for all $x\in X$.  Then for any $\epsilon>0$ there exist AF $C^*$-subalgebras $C$ and $D$ of $p(A\rtimes \Z)p$ and a positive contraction $h\in p(A\rtimes \Z)p$ such that the following hold:
\begin{enumerate}[(i)]
\item for all $x\in X$, $\|[h,x]\|<\epsilon$;
\item for all $x\in X$, $hx\epsin C$, $(1-h)x\epsin D$, $h(1-h)x\epsin C\cap D$;
\item $E:=C\cap D$ is an AF algebra;
\item $h$ multiplies $E$ into itself.
\end{enumerate}
\end{corollary}

\begin{proof}
We first follow the argument of \cite[Theorem 4.1]{Enders:2015aa}.   Let $N$ be large enough so that the conclusion of Corollary \ref{end cor} holds for the given $X$ and $p$, and parameter $\epsilon/2$, and fix any $n\geq N$.  

Note first that as $M_n(A\rtimes \Z)$ is a stable UCT Kirchberg algebra, Theorem \ref{kp stab the}, part \eqref{kp stab exist}  implies there is a $*$-isomorphism $\omega:M_n(A\rtimes \Z)\to M_n(A\rtimes \Z)$ such that the map $\omega_*:K_*(M_n(A\rtimes \Z))\to K_*(M_n(A\rtimes \Z))$ induced by $\omega$ is multiplication by $-1$ in both even and odd degrees.  

Let now $\kappa_n$ be as in Corollary \ref{end cor}, built using this $\omega$.  From Lemma \ref{enders k}, the map induced by $\kappa_n$ on $K$-theory is the same as the canonical top-left corner inclusion $i_{2n+1}:A\rtimes \Z\to M_{2n+1}(A\rtimes \Z)$, and in particular is an isomorphism on $K$-theory.  Hence by the UCT (see \cite[Proposition 7.5]{Rosenberg:1987bh} for the precise consequence of the UCT being used here), $\kappa_n$ is invertible in $KK(A\rtimes \Z,M_{2n+1}(A\rtimes \Z))$.  Theorem \ref{kp stab the}, part \eqref{kp stab exist kk} thus gives a $*$-isomorphism $\psi_n:A\rtimes \Z\to M_{2n+1}(A\rtimes \Z)$ whose class in $KK(M_{2n+1}(A\rtimes \Z),A\rtimes \Z)$ is the inverse of the class of $\kappa_n$.  

The fact that $\psi_n\circ \kappa_{n}$ equals the class of the identity in $KK(A\rtimes \Z,A\rtimes \Z)$ and Theorem \ref{kp stab the} part \eqref{kp stab unique} implies that there is a sequence $(u_m)_{m=1}^\infty$ of unitaries in the multiplier algebra of $A\rtimes \Z$ such that $u_m(\psi_n\kappa_n(a))u_m^*\to a$ as $n\to\infty$ for all $a\in A\rtimes \Z$.  

Now, let $q:=\kappa_n(p)$, and let $h_n\in q(M_{2n+1}(A\rtimes \Z))q$, and $qC_nq,qD_nq\subseteq M_{2n+1}(A\rtimes \Z)$ be as Corollary \ref{end cor}.  Define $p_m:=u_m\psi_n(q)u_m^*$ which is a projection in $A\rtimes \Z$ such that $p_m\to p$ as $m\to \infty$.  Hence by Lemma \ref{proj under} (applied to the multiplier algebra $M(A\rtimes \Z)$ of $A\rtimes \Z$) for all suitably large  $m$ there is a unitary $v_m\in M(A\rtimes \Z)$ such that $v_mp_mv_m^*=p$, and such that $v_m\to 1_{M(A\rtimes \Z)}$ as $m\to\infty$.  

Direct checks now show that for sufficiently large $m$, the element $h:=v_mu_m \psi_n(h_n)u_m^*v_m^*$ and $C^*$-subalgebras $C:=v_mu_m\psi_n(C_n)u_m^*v_m^*$, and $D:=v_mu_m\psi_n(D_n)u_m^*v_m^*$ have the properties in the statement.
\end{proof}

The next corollary follows directly from Corollary \ref{enders and kp} and the definition of complexity rank (see Definition \ref{f c} above).

\begin{corollary}\label{enders kp 2}
Let $A$ be an AF algebra equipped with an action of $\Z$ so that the associated crossed product $A\rtimes \Z$ is a Kirchberg algebra, and let $p\in A\subseteq A\rtimes \Z$ be a projection.  Then $p(A\rtimes \Z)p$ has complexity rank at most one. \qed
\end{corollary}

We are finally ready to complete the proof of Theorem \ref{kirch the 1}.  We will use corner endomorphisms and the associated crossed products by $\N$: see \cite[Section 2]{Rordam:1995aa} for background on this.

\begin{proof}[Proof of Theorem \ref{kirch the 1}]
Let $B$ be a unital Kirchberg algebra that satisfies the UCT.  Using \cite[Theorem 3.6]{Rordam:1995aa} there is a simple, unital AF algebra $A_0$ with unique trace and a proper corner endomorphism $\rho$ of $A_0$ such that the associated crossed product $A_0\rtimes \N$ is a UCT Kirchberg algebra with the same $K$-theory invariant as $B$.  Hence by the Kirchberg-Phillips classification theorem (see for example \cite[Theorem 8.4.1]{Rordam:2002cs} for an appropriate version) $B$ is isomorphic to $A_0\rtimes \N$.  Hence it suffices to prove that $A_0\rtimes \N$ has complexity rank at most one.

Define now $A$ to be the direct limit of the sequence 
$$
\xymatrix{ A_0\ar[r]^-\rho & A_0\ar[r]^-\rho & A_0\ar[r]^-\rho & \cdots}.
$$
Then $A$ is a direct limit of AF algebras so itself an AF algebra, and as discussed in \cite[pages 75-76, and also pages 72-73]{Rordam:2002cs}, $A$ is equipped with a $\Z$-action and a projection $p\in A$ such that $p(A\rtimes \Z)p\cong A_0\rtimes \N$.  Thanks to Corollary \ref{enders kp 2}, we are done.
\end{proof}

\subsection{The general case}

In this subsection, we finish the proof of Theorem \ref{kirch the} by computing the complexity rank of general unital UCT Kirchberg algebras.  We will need existence of a good class of ``models'', i.e.\ a collection of $C^*$-algebras with well-understood structure so that every UCT Kirchberg algebra is isomorphic to one in the collection.  Our models will be built from Cuntz algebras, and one other Kirchberg algebra with special $K$-theory.  We need some notation.  For $n\in \{2,3,4,...\}\cup \{\infty\}$, we let $\mathcal{O}_n$ denote the Cuntz algebra.  We also let $\mathcal{O}_{1,\infty}$ be a unital UCT Kirchberg algebra with $K_0(\mathcal{O}_{1,\infty})=0$ and $K_1(\mathcal{O}_{1,\infty})=\Z$; such exists by \cite[Proposition 4.3.3]{Rordam:2002cs} (and is unique up to isomorphism by the Kirchberg-Phillips classification theorem).

The next proposition gives the models we will use.  Variants of this are very well-known: see for example \cite[Proposition 8.4.11]{Rordam:2002cs} (and the erratum on the author's webpage).

\begin{proposition}\label{uct kirch}
Any unital Kirchberg algebra in the UCT class can be written as an inductive limit of $C^*$-algebras of the form 
\begin{equation}\label{special}
B_0\oplus (B_1\otimes \mathcal{O}_{1,\infty}),
\end{equation}
where $B_0$ and $B_1$ are both of the form 
$$
\bigoplus_{j=1}^N M_{n_j}(\mathcal{O}_{m_j}) 
$$
with $N\in \N$, each $n_j\in \N$, and each $m_j\in \{2,3,...\}\cup \{\infty\}$.
\end{proposition}

To establish this, we will need another variant of the Kirchberg-Phillips classification theorem, due to Kirchberg\footnote{As far as we are aware, Kirchberg's proof has not been published: the reader can consult \cite[Theorem A]{Gabe:2019ws} for a proof (which is independent of Kirchberg's).}.  For the statement, recall that a $*$-homomorphism $\phi:A\to B$ is \emph{full} if for any non-zero $a\in A$, $\phi(a)$ generates $B$ as a two-sided ideal.

\begin{theorem}[Kirchberg]\label{class k}
Let $A$ be a separable, nuclear, unital $C^*$-algebra that satisfies the UCT, and let $B$ be a unital, properly infinite $C^*$-algebra.  Let $[1_A]\in K_0(A)\subseteq K_*(A)$ be the class of the unit, and similarly for $[1_B]$. Then for any (graded) homomorphism $\alpha:K_*(A)\to K_*(B)$ such that $\alpha[1_A]=[1_B]$ there exists a full, unital $*$-homomorphism $\phi:A\to B$ inducing $\alpha$.
\end{theorem}

\begin{proof} 
Let $A$ be a separable, nuclear, unital $C^*$-algebra, and let $B$ be a unital, properly infinite $C^*$-algebra.  Then \cite[Theorem A]{Gabe:2019ws} implies that for any $x\in KK(A,B)$ such that the map $x_*:K_*(A)\to K_*(B)$ on $K$-theory induced by $x$ takes $[1_A]$ to $[1_B]$, there exists a full unital $*$-homomorphism $\phi:A\to B$ such that the class $[\phi]$ in $KK(A,B)$ equals $x$: precisely, the given reference has strictly weaker assumptions on $A$ and $B$ (in particular, only that $A$ is exact), and works with classes in $KK_{nuc}(A,B)$ rather than $KK(A,B)$.  However, we assume above that $A$ is nuclear, which implies that $KK_{nuc}(A,B)=KK(A,B)$.

On the other hand, as $A$ satisfies the UCT, the canonical map 
$$
KK(A,B)\to \text{Hom}(K_*(A),K_*(B))
$$ 
is surjective.  The result follows from this and the comments above on lifting $\alpha$ to some $x\in KK(A,B)$. 
\end{proof}

\begin{proof}[Proof of Proposition \ref{uct kirch}]
We note first that any $C^*$-algebra as in line \eqref{special} satisfies the UCT.  Indeed, Cuntz algebras are in the UCT class as discussed on \cite[page 73]{Rordam:2002cs}, and the UCT class is preserved under direct sums by \cite[Proposition 2.3 (a)]{Rosenberg:1987bh}, under taking matrix algebras by \cite[Proposition 2.3 (a)]{Rosenberg:1987bh}, and under taking tensor products by \cite[Theorem 7.7]{Rosenberg:1987bh}.  

Let $(K_0(A),[1_A],K_1(A))$ be the $K$-theory invariant of $A$.  Choose a sequence $(G_{n,0},G_{n,1})$ such that $G_{n,0}$ is a finitely generated subgroup of $K_0(A)$ containing $[1_A]$, $G_{n,1}$ is a finitely generated subgroup of $K_1(A)$, and such that $K_i(A)=\bigcup_{n\in \N}G_{n,i}$ for $i\in \{0,1\}$.  As $C^*$-algebras of the form in line \eqref{special} and the various building blocks involved satisfy the UCT, the $K$-theory K\"{u}nneth formula applies (see \cite[page 443]{Schochet:1982aa} or \cite[Theorem 23.1.3]{Blackadar:1998yq}).  Using this and the well-known $K$-theory of the Cuntz algebras (see for example \cite[page 74]{Rordam:2002cs}) it is straightforward to see that for each $n$, there is a $C^*$-algebra $C_n$ of the form in line \eqref{special} such that $(K_0(C_n),[1_{C_n}],K_1(C_n))\cong (G_{0,n},[1_A],G_{1,n})$.   Identifying these groups via a fixed isomorphism, Corollary \ref{class k} implies that for each $n$ the inclusion map 
$$
(G_{0,n},[1_A],G_{1,n})\to (G_{0,n+1},[1_A],G_{1,n+1})
$$
is induced by a full unital $*$-homomorphism $\phi_n:C_n\to C_{n+1}$.  We claim that $A$ is isomorphic to the inductive limit $C$ of the system $(C_n,\phi_n)$.  Indeed, as each $\phi_n$ is unital and full, $C$ is unital and simple. Using continuity of $K$-theory, $(K_0(C),[1_C],K_1(C))\cong (K_0(A),[1_A],K_1(A))$.    As each $C_n$ is nuclear, $C$ is nuclear (see for example \cite[Theorem 10.1.5]{Brown:2008qy}).  As each $C_n$ is a finite direct sum of purely infinite $C^*$-algebras, $C$ is purely infinite (one can check this using the condition in \cite[Proposition 4.1.8 (iv)]{Rordam:1995aa}, for example).  As each $C_n$ satisfies the UCT, $C$ also satisfies the UCT by \cite[Proposition 2.3 (b)]{Rosenberg:1987bh}.  Hence by the Kirchberg-Phillips classification theorem (for example, in the form of \cite[Theorem 4.3.4]{Phillips-documenta}), $A$ is isomorphic to $C$ as claimed. 
\end{proof}

\begin{theorem}\label{kirch cr2}
Any unital UCT Kirchberg algebra $A$ has complexity rank at most two.  
\end{theorem}

\begin{proof}
Proposition \ref{uct kirch} writes $A$ as an inductive limit of $C^*$-algebras of the form $B_0\oplus (B_1\otimes \mathcal{O}_{1,\infty})$ with $B_0$ and $B_1$ a finite direct sum of matrix algebras over Cuntz algebras.  Using Theorem \ref{kirch the 1}, any unital UCT Kirchberg algebra with torsion free $K_1$-group has complexity rank one.  Using this and Lemma \ref{sums}, each of $B_0$, $B_1$ and $\mathcal{O}_{1,\infty}$ has complexity rank at most one.  Hence Proposition \ref{perm rem} implies that $B_1\otimes \mathcal{O}_{1,\infty}$ has complexity rank at most two, and thus so does $B_0\oplus B_1\otimes \mathcal{O}_{1,\infty}$ using Lemma \ref{sums} again.  As complexity rank is non-increasing under taking inductive limits (Lemma \ref{loc in}), the complexity rank of $A$ is at most two.
\end{proof}

We finish this section by recording a proof of Theorem \ref{kirch the}.

\begin{proof}[Proof of Theorem \ref{kirch the}]
Let $A$ be a unital UCT Kirchberg algebra.  Then $A$ has complexity rank at most two by Theorem \ref{kirch cr2}.  As $A$ is not locally finite-dimensional, it does not have complexity rank zero.  

If $A$ has complexity rank one, then it has torsion-free $K_1$-group by Theorem \ref{k1tf}.  Conversely, if $A$ has torsion-free $K_1$ group, then it has complexity rank one by Theorem \ref{kirch the 1}.
\end{proof}

\section{Questions}\label{question sec}

We conclude the paper with some open questions that seem interesting to us.

The first question is important (and probably difficult) as it is equivalent to the UCT for all nuclear $C^*$-algebras.

\begin{question}
Do all (unital) Kirchberg algebras have finite complexity?
\end{question}

Even knowing finite complexity for Kirchberg algebras with trivial $K$-theory would imply the UCT for all nuclear $C^*$-algebras.

The next question is about the most interesting example that we do not currently know the complexity rank of.  

\begin{question}
What is the complexity rank of an irrational rotation algebra?  
\end{question}

We conjecture the answer is always one; more generally, we conjecture that the complexity rank of a separable $A\mathbb{T}$-algebra of real rank zero (and which is not AF) is always one.

\begin{question}
What is the complexity rank of (classifiable) AH (or even ASH) algebras of real rank zero? 
\end{question}

It would also be interesting to give non-trivial upper bounds, maybe in terms of the dimensions of the spectra of  (sub)homogeneous algebras appearing in a directed system for the given A(S)H algebra.

The following question is very natural.  We know too little to hazard a reasonable guess at the moment.

\begin{question}
Which ordinal numbers can be the complexity rank of a $C^*$-algebra?
\end{question}

We did not seriously attempt to address this question, but at the moment, the only values we know can be taken are $0$, $1$, and $2$.   It is conceivable that for the uniform Roe algebras $C^*_u(X)$ associated to a space $X$, the complexity rank of $C^*_u(X)$ and the complexity rank of $X$ in the sense of \cite[Definition 2.9]{Guentner:2013aa} coincide (at present we know only that the complexity rank of the $C^*$-algebra is bounded above by that of the space).  If these ranks were equal, it would follow for example from \cite[Sections 4 and 5]{Guentner:2013aa} and \cite{Chen:2015uc} that many complexity ranks are possible for $C^*$-algebras.

\begin{question}
Does (weak) complexity rank at most one imply real rank zero in general?  
\end{question}

There are some interesting connections of this question to other problems: compare Remark \ref{gen rr0 rem} above. 

The question below seems interesting from the point of view of the structure of $C^*$-algebras.  Recall from the discussion below Definition \ref{intro dec} that we think of having weak complexity rank at most one as being `two-colored locally finite-dimensional'. 

\begin{question}
Can one make a reasonable version of being `two-colored AF' that is also equivalent to having weak complexity rank at most one?
\end{question} 

This would mean somehow arranging the different $C^*$-subalgebras $C$ and $D$ that arise into systems ordered by inclusion in some sense (to be made precise by the answer to the question!).  By analogy with the classical (non-)equivalence between being AF and being locally finite-dimensional, one probably wants to assume separability in the above.

The following question seems basic (we tried to find an answer and were not able to).

\begin{question}
Does having complexity rank at most $\alpha$ pass to corners?   
\end{question}

This would be interesting to know even for $\alpha=1$.  The answer is `yes' for \emph{weak} complexity rank at most one: one can see this by adapting the proof of \cite[Proposition 3.8]{Kirchberg:2004uq}, for example.  

Our last question is a little vague, but would be useful to have, particularly with regards permanence properties.

\begin{question}
Is there a `good' definition of decomposability in the non-unital case?
\end{question}

Many of the results in this paper have reasonably natural variants in the non-unital case, but we were not able to come up with a really clean and natural definition, so in the end opted to write the paper entirely in the unital setting for the sake of simplicity.  Certainly having a notion that applied equally in the unital case would be very interesting, however.

\bibliography{Generalbib}

\end{document}